\documentclass[final]{siamart1116}

\usepackage{lipsum}
\usepackage{amsfonts}
\usepackage{graphicx}
\usepackage{epstopdf}
\usepackage{algorithmic}
\ifpdf
  \DeclareGraphicsExtensions{.eps,.pdf,.png,.jpg}
\else
  \DeclareGraphicsExtensions{.eps}
\fi

\numberwithin{theorem}{section}

\newcommand{\TheTitle}{Positivity, monotonicity, and consensus on Lie Groups} 
\newcommand{\TheAuthors}{C. Mostajeran, and R. Sepulchre}

\headers{\TheTitle}{\TheAuthors}

\title{{\TheTitle}\thanks{This work was funded by the Engineering and Physical Sciences Research Council (EPSRC) of the United Kingdom, as well as the European Research Council under the Advanced ERC Grant Agreement Switchlet n.670645.}}

\author{
  Cyrus Mostajeran
  \thanks{Department of Engineering, University of Cambridge, United Kingdom (\email{csm54@cam.ac.uk}, \email{r.sepulchre@eng.cam.ac.uk}).}
    \and   Rodolphe Sepulchre \footnotemark[2]
}

\usepackage{amsopn}
\DeclareMathOperator{\diag}{diag}

\ifpdf
\hypersetup{
  pdftitle={\TheTitle},
  pdfauthor={\TheAuthors}
}
\fi

\newtheorem{remark}{Remark}

\begin{document}

\maketitle

\begin{abstract}
Dynamical systems whose linearizations along trajectories are positive in the sense that they infinitesimally contract a smooth cone field are called differentially positive. The property can be thought of as a generalization of monotonicity, which is differential positivity in a linear space with respect to a constant cone field. Differential positivity places significant constraints on the asymptotic behavior of trajectories under mild technical conditions. This paper studies differentially positive systems defined on Lie groups. The geometry of a Lie group allows for the generation of invariant cone fields over the tangent bundle given a single cone in the Lie algebra. We outline the mathematical framework for studying differential positivity of 
discrete and continuous-time dynamics on a Lie group with respect to an invariant cone field and motivate the use of this analysis framework in nonlinear control, and, in particular in nonlinear consensus theory. We also introduce a generalized notion of differential positivity of a dynamical system with respect to an extended notion of cone fields generated by cones of rank $k$. This new property provides the basis for a generalization of differential Perron-Frobenius theory, whereby the Perron-Frobenius vector field which shapes the one-dimensional attractors of a differentially positive system is replaced by a distribution of rank $k$ that results in $k$-dimensional integral submanifold attractors instead. 
\end{abstract}



\section{Introduction}

Positivity and monotonicity play an important role in the study of dynamical systems. Both concepts have been the subject of considerable interest in recent years because of their importance in the convergence analysis of consensus algorithms \cite{Moreau2004,Olfati-Saber2006,Olfati-Saber2007,Jadbabaie2003,Sepulchre2008} and in the modeling of biological systems \cite{Enciso2005,Angeli2008,Angeli2012}. While both concepts have been extensively studied in vector spaces, the present paper seeks to extend the theory to dynamical systems defined on Lie groups. An important motivation is to highlight the similar properties of consensus algorithms defined on vector spaces and on Lie groups  \cite{Sarlette2010, Tron2012}. While positivity is at the core of consensus algorithms studied in vector spaces, the concept has not been used so far in the convergence analysis of consensus algorithms defined on Lie groups, preventing the generalization of important results such as the consideration of inhomogeneous dynamics, asymmetric couplings, and time-varying network topologies in the problem formulation.\footnote{This manuscript further develops ideas partially covered in \cite{Mostajeran2016,Mostajeran2017}}

The proposed extension is through the concept of differential positivity, recently introduced in \cite{Forni2015}. Linear positive systems are systems that leave a cone invariant \cite{Bushell1973}. Such systems find many applications in control engineering, including to stabilization \cite{Muratori1991,Farina2000}, observer design \cite{Bonnabel2011}, and distributed control \cite{Moreau2004}. An important feature of positivity is that it restricts the behavior of a system, as seen in Perron-Frobenius theory. To illustrate these ideas, let the vector space $\mathcal{V}$ be the state space of the system and consider the linear dynamics $\dot{x}=Ax$ on $\mathcal{V}$. Such a system is said to be positive with respect to a pointed convex solid cone $\mathcal{K}\subseteq \mathcal{V}$ if 
$e^{At}\mathcal{K}\subseteq\mathcal{K}$,
for all $t>0$, where $e^{At}\mathcal{K}:=\{e^{At}x:x\in\mathcal{K}\}$. Perron-Frobenius theory demonstrates that if the system is strictly positive in the sense that the transition map $e^{At}$ maps the boundary of the cone $\mathcal{K}$  into its interior, then any trajectory $e^{At}x$, $x\in\mathcal{K}$ converges asymptotically to the subspace spanned by the unique dominant eigenvector of $A$. Differentially positive systems are systems whose linearizations  along trajectories are positive. Strictly differentially positive systems infinitesimally contract a cone field along trajectories, constraining the asymptotic behavior to be one-dimensional under suitable technical conditions.

The study of differential positivity of a system defined on a manifold necessitates the construction of a cone field which assigns to each point a cone that lies in the tangent space at that point. On highly symmetric spaces such as Lie groups, which are commonly found in applications, there is a strong incentive to make the differential analysis {\it invariant} by incorporating the symmetries of the state-space in the analysis. 
Monotone dynamical systems are defined as systems that preserve a partial order in vector spaces. An infinitesimal characterization of the monotonicity is differential positivity with respect to a constant cone field identified with the cone defined in the vector space and associated to the partial order. The infinitesimal characterization suggests a natural generalization to Lie groups, requiring differential positivity with respect to an invariant cone field.

We also consider systems that are positive with respect to so-called cones of rank $k\geq 2$. These structures are a generalization of classical solid convex cones and lead to a weakened notion of monotonicity. For linear systems that are strictly positive with respect to a cone of rank $k\geq 2$, the one-dimensional dominant eigenspace of Perron-Frobenius theory is replaced by an eigenspace of dimension $k$; see \cite{Fusco1991}.
Cones of rank $k\geq 2$ have been used to study the existence of periodic orbits via a Poincare-Bendixson property for a new class of monotone systems (as in \cite{smith1980,sanchez2009,sanchez2010}), including for cyclic feedback systems \cite{malletparet1990,malletparet1996}. Here we develop a differential analytic formulation of positivity with respect to a cone field of rank $k\geq 2$, which is applicable to the study of nonlinear systems. We discuss how the property leads to a generalization of differential Perron-Frobenius theory, whereby the attractors of the system are shaped by a distribution of rank $k$.

\subsection{Contributions} In short, the present paper exploits the concept of invariance in differential positivity and uses invariant differential positivity to generalize consensus theory from linear spaces to Lie groups. Specifically, the paper focuses on three original contributions. The first contribution is the formulation of invariant differential positivity on Lie groups. The second contribution is to further generalize the concept by considering higher rank cone fields, mimicking the corresponding generalization of monotone systems. The third contribution is to apply the proposed theory to consensus algorithms defined on the circle and to $SO(3)$, highlighting the benefits of a convergence theory rooted in positivity rather than in stability.

\subsection{Paper organization} We begin with a review of differential positivity and its relation to monotonicity and consensus in linear spaces. In Section \ref{Lie}, we study the geometry of cone fields and formulate invariant differential positivity on Lie groups. We also provide a generalization of existing results on linear positivity to nonlinear systems that are differentially positive with respect to cone fields of rank $k$. 
In Section \ref{circle}, invariant differential positivity is applied to a class of consensus protocols defined on the $N$-torus $\mathbb{T}^N$. 
Finally, in Section \ref{SO(3)}, we study an extended example involving consensus of $N$ agents on $SO(3)$ through differential positivity with respect to a cone field of rank $k=3$ that is invariant with respect to a left-invariant frame on $SO(3)^N$. 

\subsection{Mathematical preliminaries and notation} 

Given a smooth map $F:\mathcal{M}_1\rightarrow\mathcal{M}_2$ between smooth manifolds $\mathcal{M}_1$, $\mathcal{M}_2$, we denote the differential of $F$ at $x$ by $dF\vert_x : T_x\mathcal{M}_1\rightarrow T_{F(x)}\mathcal{M}_2$. 
A continuous-time dynamical system $\Sigma:\dot{x}=f(x)$ on a smooth manifold $\mathcal{M}$ assigns a tangent vector $f(x)\in T_x\mathcal{M}$ to each point $x\in\mathcal{M}$. We say that $\Sigma$ is  \emph{forward complete} if the domain of any solution $x(\cdot)$ 
is of the form $[t_0,\infty)$.

\subsubsection{Cones of rank $k$ and positivity}

 A cone in a vector space $\mathcal{V}$ is usually defined as a closed subset $\mathcal{K}\subset \mathcal{V}$ that satisfies
(i) $\mathcal{K}+\mathcal{K}\subseteq\mathcal{K}$,
(ii) $\lambda \mathcal{K}\subseteq\mathcal{K}$ for all $\lambda\in\mathbb{R}_{\geq 0}$, and
(iii) $\mathcal{K}\cap-\mathcal{K}=\{0\}$.
That is, $\mathcal{K}$ is closed, convex, and pointed. Furthermore, we assume that we are dealing with solid cones that contain $n:=\dim\mathcal{V}$ linearly independent vectors. However, we also use the notion of cones of rank $k$, the simplest case of which is a cone of rank $1$. Given a convex cone $\mathcal{K}$, the set $\mathcal{C}=\mathcal{K}\cup-\mathcal{K}$  defines a cone of rank $1$. The full definition is given below.

\begin{definition}
A closed set $\mathcal{C}$ in a vector space $\mathcal{V}$ is said to be a cone of rank $k$ if 
\begin{enumerate}
\item $x\in\mathcal{C}$, $\alpha\in\mathbb{R} \implies \alpha x\in\mathcal{C}$,
\item $\max\{\dim W:W$ a subspace of $\mathcal{V}, \; W\subset \mathcal{C}\}=k$.
\end{enumerate}
\end{definition}
That is, a closed set $\mathcal{C}$ in a vector space $\mathcal{V}$ is a cone of rank $k$ if (1) it is invariant under scaling by any real number, and (2) the maximum dimension of any subspace contained in $\mathcal{C}$ is $k$.
Note that if $\mathcal{K}$ is a convex cone, $\mathcal{C}=\mathcal{K}\cup-\mathcal{K}$ satisfies the above conditions for $k=1$. 

A polyhedral convex cone in a vector space $\mathcal{V}$ of dimension $n$ endowed with an inner product $\langle\cdot,\cdot\rangle$ can be specified by a collection of inequalities of the form
\begin{equation} \label{polyhedral}
x\in\mathcal{V}: \quad \langle n_i,x\rangle\geq  0, 
\end{equation}
where $\{n_1,\cdot\cdot\cdot, n_m\}$ is a set of $m\geq n$ vectors in $\mathcal{V}$ that span $\mathcal{V}$. For each $i$, the equation (\ref{polyhedral}) defines a halfspace defined by the normal vector $n_i\in\mathcal{V}$. If we relax the requirement that $\{n_i\}$ span $\mathcal{V}$ and instead require that $\dim\operatorname{span}\{n_i\}=l\leq n$, the collection of inequalities (\ref{polyhedral}) define a convex set $\tilde{\mathcal{K}}$ of rank $k=n-l+1$, such that the set
$
\mathcal{C}:=\tilde{\mathcal{K}}\cup-\tilde{\mathcal{K}}
$
is a cone of rank $k=n-l+1$. In particular, $n-k+1$ inequalities of the form $(\ref{polyhedral})$, with linearly independent $n_i$, can be used to define a cone of of rank $k$. A very simple example of a class of polyhedral cones of rank $k$ is obtained from the positive orthant $\mathcal{K}=\mathbb{R}_+^n=\{x_i:x_i \geq0\}$ in $\mathbb{R}^n$ by eliminating $k-1$ of the inequalities $x_i\geq 0$ and retaining the remaining ones. The resulting set $\tilde{\mathcal{K}}$ can be used to generate a generalized cone of rank $k$ as $\mathcal{C}=\tilde{\mathcal{K}}\cup-\tilde{\mathcal{K}}$.

A second class of cones of rank $k$ can be defined using quadratic forms. These cones are known as quadratic cones of rank $k$ and are a generalization of the idea of quadratic cones of rank $1$. Let $P$ be a symmetric  invertible $n\times n$ matrix with $k$ positive eigenvalues and $n-k$ negative eigenvalues. Then the set
\begin{equation}
\mathcal{C}(P)=\{x\in\mathcal{V}:\langle x,Px\rangle\geq0\},
\end{equation}
can be shown to define a cone of rank $k$. In particular, if $P_1$ is a $k\times k$ symmetric positive definite matrix and $P_2$ is an $(n-k)\times(n-k)$ symmetric positive definite matrix, the $n\times n$ matrix 
$P = \diag(P_1,-P_2)$ has $k$ positive eigenvalues and $n-k$ negative eigenvalues and defines a cone of rank $k$ in $\mathbb{R}^n$ via the inequality
\begin{align}
x^TPx &=
\begin{pmatrix}
x_1 \\
x_2
\end{pmatrix}^T
\begin{pmatrix}
P_1 & 0 \\
0 & -P_2
\end{pmatrix}
\begin{pmatrix}
x_1 \\
x_2
\end{pmatrix} =
x_1^T P_1 x_1 - x_2^T P_2 x_2 \geq 0,
\end{align}
where $x_1\in\mathbb{R}^k$, $x_2\in\mathbb{R}^{n-k}$. Given a quadratic cone $\mathcal{C}$ of rank $k$, the closure of the set $\mathcal{V}\setminus\mathcal{C}$ is a cone of rank $n-k$, and is known as the complementary cone $\mathcal{C}^c$ of $\mathcal{C}$.

In this article, a cone $\mathcal{C}$ will refer to either a closed, convex, and pointed cone $\mathcal{K}$, or a cone of rank $k$. The precise nature of the cone will be made explicit whenever it is relevant to the result under consideration.
A linear map $T\in\mathcal{L}(\mathcal{V})$ on a vector space $\mathcal{V}$ is said to be positive with respect to a cone $\mathcal{C}\subset \mathcal{V}$ if 
$T(\mathcal{C})\subseteq\mathcal{C}$. The classical Perron-Frobenius theorem generalizes to cones of rank $k$ in the form of the following result \cite{Fusco1991} .

\begin{theorem}  \label{PF k}
Let $\mathcal{C}$ be a cone of rank $k$ in a vector space $\mathcal{V}$ of dimension $n$ and suppose that $T\in\mathcal{L}(\mathcal{V})$ is a strictly positive linear map with respect to $\mathcal{C}$ so that
$T(\mathcal{C}\setminus\{0\})\subset\operatorname{int}\mathcal{C}$. Then there exist unique subspaces $\mathcal{W}_1$ and $\mathcal{W}_2$ of $\mathcal{V}$ such that 
$\dim\mathcal{W}_1=k$, $\dim\mathcal{W}_2=n-k$, $\mathcal{V}=\mathcal{W}_1\oplus\mathcal{W}_2$, which are $T$-invariant:
\begin{equation}
T(\mathcal{W}_i)\subseteq\mathcal{W}_i, \quad \mathrm{for} \quad i=1,2,
\end{equation}
and satisfy $
\mathcal{W}_1\subset \operatorname{int}\mathcal{C}\cup\{0\}$, $\mathcal{W}_2\cap\mathcal{C}=\{0\}$.
Furthermore, denoting the spectrum of $T$ restricted to $\mathcal{W}_i$ by $\sigma_i(T)$ for $i=1,2$ , we have
\begin{equation}
|\lambda_1|>|\lambda_2|, \quad \forall \lambda_1\in\sigma_1(T), \; \lambda_2\in\sigma_2(T).
\end{equation}
\end{theorem}

\subsubsection{Conal manifolds}

A pointed convex cone $\mathcal{K}$ in $\mathbb{R}^n$ induces a partial order in $\mathbb{R}^n$ such that for any pair $x_1,x_2\in \mathbb{R}^n$, we write
$
x_1\preceq x_2$ if and only if $x_2-x_1\in\mathcal{K}$.
It is of course the convexity of $\mathcal{K}$ and the condition $\mathcal{K}\cap-\mathcal{K}=\{0\}$ which ensure that $\preceq$ defines a global partial order on $\mathbb{R}^n$. As a cone $\mathcal{C}$ of rank $k$ in $\mathbb{R}^n$ does not satisfy these conditions, it does not induce a partial order. Nonetheless, we can still define a weakened notion of an `order' relation between two points $x_1,x_2\in\mathbb{R}^n$ with respect to $\mathcal{C}$ in a similar fashion. Specifically, we say that $x_1$ and $x_2$ are \emph{related} with respect to $\mathcal{C}$ and write $x_1\sim x_2$ if 
$
x_2-x_1\in\mathcal{C}$.
We say that $x_1$ and $x_2$ are strongly related and write $x_1\approx x_2$ if $x_2-x_1\in\operatorname{int}\mathcal{C}$.  

We define a \emph{conal manifold} as a smooth manifold $\mathcal{M}$ endowed with a \emph{cone field} $\mathcal{C}_{\mathcal{M}}$, which smoothly assigns a cone
$\mathcal{C}_{\mathcal{M}}(x)\subseteq T_x\mathcal{M}$
to each point $x\in\mathcal{M}$.
Given a cone field $\mathcal{C}$ on a manifold $\mathcal{M}$, we say that two points $x_1, x_2\in\mathcal{M}$ are related with respect to $\mathcal{C}$ and write $x_1\sim x_2$ if there exists a conal curve $\gamma$ connecting $x_1$ to $x_2$, so that
$
\gamma'(t)\in\mathcal{C}(\gamma(t))$
at all points along the curve.
Indeed, if $\mathcal{C}$ is a cone field of rank $k$ on a manifold $\mathcal{M}$, the continuous-time system $\dot{x}=f(x)$ on $\mathcal{M}$ with semiflow $\psi_t(x)$ is said to be \emph{monotone} with respect to $\mathcal{C}$ if
\begin{equation}
x_1\sim x_2 \implies \psi_t(x_1)\sim\psi_t(x_2),
\end{equation}
for all $t>0$. The system is \emph{strongly monotone} with respect to $\mathcal{C}$ if
\begin{equation}
x_1\sim x_2, \quad x_1\neq x_2 \implies \psi_t(x_1)\approx\psi_t(x_2),
\end{equation}
for all $t>0$.

 A field $\mathcal{K}_{\mathcal{M}}$ of closed, convex and pointed cones gives rise to a \emph{conal order}  $\prec$ and the manifold $\mathcal{M}$ is said to be an \emph{infinitesimally partially ordered} manifold when endowed with $\mathcal{K}_{\mathcal{M}}$. The relation $\sim$ induced by $\mathcal{K}_{\mathcal{M}}$ is now denoted by $\prec$ and is locally a partial order on $\mathcal{M}$. If the conal order is globally antisymmetric, then it is a partial order. It is clear that $\prec$ defines a global partial order when $\mathcal{M}$ is a vector space and the cone field $\mathcal{K}_{\mathcal{M}}(x)=\mathcal{K}$ is constant. Specifically, $x\prec y$ if and only if $y-x\in\mathcal{K}$, for $x$, $y\in\mathcal{M}$. In general, however, $\prec$ is not a global partial order, since antisymmetry may fail. A simple example of this is provided by any conal order defined on the circle $\mathbb{S}^1$, which clearly fails to be global. See Figure~\ref{fig:ConeOrder} $(a)$.

\begin{figure}
\centering
\includegraphics[width=1\linewidth]{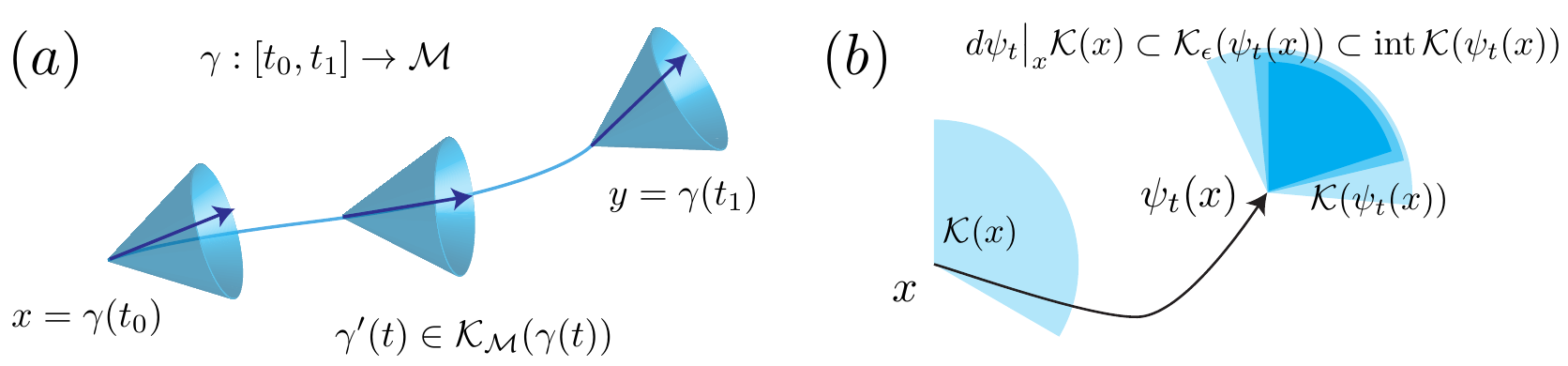}
  \caption{$(a)$ A conal order $\prec$ induced on a manifold $\mathcal{M}$ by a cone field $\mathcal{K}_{\mathcal{M}}$. Points $x,y\in\mathcal{M}$ satisfy $x\prec y$ if there exists a conal curve $\gamma$ from $x$ to $y$. $(b)$ Strict differential positivity of the system $\dot{x}=f(x)$ with respect to a cone field $\mathcal{K}$. The flow at time $t\geq T$ from initial condition $x$ is denoted by $\psi_t(x)$.
  }
  \label{fig:ConeOrder}
\end{figure}

\subsubsection{Differential positivity} 

To define the notion of \emph{uniform} strict differential positivity on a conal manifold $\mathcal{M}$, we construct a family of cones $\mathcal{C}_{\lambda}(x)\subset T_x\mathcal{M}$ of the same type at each point $x\in\mathcal{M}$ that continuously depend on a parameter $\lambda\geq 0$, such that $\mathcal{C}_0(x)=\mathcal{C}(x)$ and $\mathcal{C}_{\lambda_2}(x)\setminus\{0_x\}\subset \operatorname{int}\mathcal{C}_{\lambda_1}(x)$ if $\lambda_1<\lambda_2$. See Figure~\ref{fig:ConeOrder} $(b)$.

\begin{definition}
A discrete-time dynamical system $x^+=F(x)$ given by a smooth map $F:\mathcal{M}\rightarrow\mathcal{M}$ on a manifold $\mathcal{M}$ is said to be differentially positive with respect to a cone field $\mathcal{C}$ if 
\begin{equation}  \label{discrete pos}
d F\big\vert_x\mathcal{C}(x)\subseteq\mathcal{C}(F(x)), \quad \forall x\in\mathcal{M},
\end{equation}
and strictly differentially positive if the inclusion in (\ref{discrete pos}) is strict.
A continuous-time dynamical system $\Sigma$ is said to be differentially positive with respect to $\mathcal{C}$ if 
\begin{equation}
d\psi_t\big\vert_x\mathcal{C}(x)\subseteq\mathcal{C}(\psi_t(x)). \quad \forall x\in\mathcal{M}, \; \forall t\geq 0,
\end{equation}
where $\psi_t(x)$ denotes the flow at time $t$ from initial condition $x$.
The system is said to be uniformly strictly differentially positive if there exists $T>0$ and $\epsilon>0$ such that
\begin{equation}
d\psi_t\big\vert_x\mathcal{C}(x)\subset\mathcal{C}_{\epsilon}(\psi_t(x)), \quad \forall x\in\mathcal{M}, \; \forall t\geq T. 
\end{equation}
\end{definition}

 \subsubsection{Finsler metrics} 
 
We define a \emph{Finsler metric} on a smooth manifold $\mathcal{M}$ to be a continuous function $F:T\mathcal{M}\rightarrow [0,\infty]\subset\mathbb{R}$ that satisfies the following properties:
\begin{enumerate}
\item $F(x,\xi) > 0$ for all $x\in \mathcal{M}$ and $\xi\in T_x\mathcal{M}\setminus\{0_x\}$,
\item $F(x,\xi_1+\xi_2)\leq F(x,\xi_1)+F(x,\xi_2)$, for all $x\in \mathcal{M}$, $\xi_1,\xi_2\in T_x\mathcal{M}$,
\item $F(x,\lambda\xi)=\lambda F(x,\xi)$, for all $x\in T_x\mathcal{M}$, $\xi\in T_x\mathcal{M}$ and $\lambda \geq 0$.
\end{enumerate}
A Finsler metric is \emph{reversible} if it satisfies $F(x,-\xi)=F(x,\xi)$ for all $x\in\mathcal{M}$ and $\xi\in T_x\mathcal{M}$. A reversible Finsler metric defines a norm on each tangent space, which we will denote by $F(x,\xi)=\|\xi\|_x$ when the choice of the Finsler structure $F$ is clear from the context. 
A reversible Finsler structure defines a distance function $d$ between points $x_1,x_2\in\mathcal{M}$ according to
\begin{equation}
d(x,y)=\inf_{\gamma}\int_0^1\|\gamma'(t)\|_{\gamma(t)}dt,
\end{equation}
where the infimum is taken over all smooth curves $\gamma:[0,1]\rightarrow \mathcal{M}$ joining $x_1$ and $x_2$.  Clearly, a Riemannian metric tensor  on $\mathcal{M}$ given by a smoothly varying inner product $\langle\cdot,\cdot\rangle_x$ induces a  reversible Finsler structure on $\mathcal{M}$.

\section{Positivity, monotonicity, and consensus on the real line} \label{real}

In this section, we assume that the cone $\mathcal{K}$ is closed, convex, and pointed.
The special case of differential positivity on a linear space with respect to a constant cone field  highlights that invariant differential positivity in a linear space is precisely the local characterization of monotonicity. Indeed, recall that a dynamical system $\Sigma$ on a vector space $\mathcal{V}$ endowed with a partial order $\preceq$ induced by some cone  $\mathcal{K}\subseteq\mathcal{V}$ is said to be monotone if for any $x_1, x_2\in\mathcal{V}$ the trajectories $\psi_{t}$ satisfy
$\psi_{t}(x_1)\preceq_{\mathcal{K}}\psi_{t}(x_2)$ whenever $x_1\preceq_{\mathcal{K}}x_2$,
for all $t > 0$. Now endow the manifold $\mathcal{V}$ with the constant cone field 
$\mathcal{K}_{\mathcal{V}}(x):=\mathcal{K}$ and note that the infinitesimal difference $\delta x(\cdot):=\hat{x}(\cdot)-x(\cdot)$ between two ordered neighboring solutions $x(t)\preceq_{\mathcal{K}_{\mathcal{V}}}\hat{x}(t)$ satisfies $\delta x(t)\in\mathcal{K}_{\mathcal{V}}(x(t))$, $\forall t\geq t_0$. Since $(x(\cdot),\delta x (\cdot))$ is a trajectory of the prolonged or variational system $\delta \Sigma$, this shows that the system is monotone if and only if it is differentially positive.
That is, the system is monotone if and only if for all $t > 0$,
\begin{equation}
\delta x(0)\in\mathcal{K} \implies \delta x(t)\in\mathcal{K}.
\end{equation}
A linear space is of course a  Lie group with the group operation given by linear translations. A constant cone field thus has the direct interpretation of a cone field first defined at identity and then translated to every point in an invariant manner. See Section \ref{Lie} for details.

Discrete-time linear consensus algorithms result in time-varying systems of the form
$
x^+(t):= x(t+1)=A(t)x(t),
$
where $A(t)$ is row-stochastic. That is, its elements are non-negative and its rows sum to $1$:
$A(t)\boldsymbol{1}=\boldsymbol{1}$, and $a_{ij}(t)\geq 0$,  for $i\neq j$,
where $\boldsymbol{1}=(1,\ldots,1)^T$. Protocols of this form arise from $N$ nodes exchanging information about a scalar quantity $x_i(t)$ along communication edges $(i,j)$ weighted by positive scalars $a_{ij}(t)>0$:
$
x^+_{k}(t)=\sum_{i:(k,i)\in\mathcal{E}}a_{ki}(t)x_i(t),
$
where $\mathcal{G}=(\mathcal{V},\mathcal{E})$ is the communication graph with vertices given by the set $\mathcal{V}$ and edges given by $\mathcal{E}$. For vertices $i$ and $j$, we set $a_{ij}=0$ if and only if $(i,j)\notin\mathcal{E}$. Recall that a directed graph or \emph{digraph} consists of a finite set $\mathcal{V}$ of vertices and a set $\mathcal{E}$ of edges which represent interconnections among the vertices expressed as ordered pairs $(i,j)$ of vertices. A weighted digraph is a digraph together with a set of weights that assigns a nonnegative scalar $a_{ij}$ to each edge $(i,j)$. A digraph is said to be \emph{undirected} if $a_{ij}=a_{ji}$ for all $i,j\in\mathcal{V}$. If $(i,j)\in\mathcal{E}$ whenever $(j,i)\in\mathcal{E}$, but perhaps $a_{ij}\neq a_{ji}$ for some $i,j\in\mathcal{V}$, then the graph is said to be \emph{bidirectional}. A digraph is said to be \emph{strongly connected} if there exists a directed path from every vertex to every other vertex. One can also consider a time-varying graph $\mathcal{G}(t)$ in which the vertices remain fixed, but the edges and weights can depend on time. A time-varying digraph is called a $\delta$-digraph if the nonnegative weights $a_{ij}(t)$ are bounded and satisfy $a_{ij}(t)\geq \delta >0$, for all $(i,j)\in\mathcal{E}(t)$.

Tsitsiklis \cite{Tsitsiklis1984} observed that the Lyapunov function
\begin{equation} \label{consensus2}
V(x)=\max_{1\leq i \leq N}x_i-\min_{1\leq i\leq N} x_i
\end{equation}
is never increasing along solutions of the consensus dynamics. Under suitable connectedness assumptions, it decreases uniformly in time. The non-quadratic nature of the Lyapunov function (\ref{consensus2}) is a key feature in the analysis of consensus algorithms. This property is intimately connected to the Hilbert metric. It is easy to see that the linear system $x^+=Ax$ is strictly positive monotone with respect to the positive orthant $\mathbb{R}^N_+$ in $\mathbb{R}^N$ for a strongly connected graph $\mathcal{G}$. The uniform strict positivity also holds in the case of a time-varying consensus protocol, provided that the digraph $\mathcal{G}(t)$ is a $\delta$-digraph that is uniformly connected over a finite time horizon. See  \cite{Moreau2004,Moreau2005} for a definition.

Given a cone $\mathcal{K}$ in $\mathbb{R}^N$, the Hilbert metric $d_{\mathcal{K}}$ induced by $\mathcal{K}$ 
defines a projective metric in $\mathcal{K}$ \cite{Bushell1973}. Birkhoff's theorem establishes a link between positivity and contraction of the Hilbert metric \cite{Birkhoff1957}, essentially providing a projective fixed point theorem based on the contraction of this metric. Indeed, since $A\boldsymbol{1}=\boldsymbol{1}$, Birkhoff's theorem suggests that the projective distance to $\boldsymbol{1}$ is a natural choice of Lyapunov function:
\begin{equation}
V_{B}(x)=d_{\mathbb{R}^N_+}(x,\boldsymbol{1})=\log{\max_i x_i \over \min_i x_i} = \max_i\log x_i - \min_i\log x_i,
\end{equation}
which is clearly the same as the Tsitsiklis Lyapunov function in $\log$ coordinates. 

Similarly, continuous-time linear consensus algorithms take the form
$
\dot{x}=A(t)x,
$
where $A(t)$ is now a \emph{Metzler} matrix whose rows sum to zero and whose off-diagonal elements are non-negative:
$A(t)\boldsymbol{1}=0$, and  $a_{ij}(t)\geq 0$ for $i\neq j$.
Such continuous-time linear protocols arise from dynamics of the form
$
\dot{x}_{k}=\sum_{i:(k,i)\in \mathcal{E}}a_{ki}(t)(x_i-x_k),
$
and are once again uniformly strictly differentially positive with respect to the positive orthant $\mathcal{K}:=\mathbb{R}^N_+$ in $\mathbb{R}^N$ for a strongly connected time-varying digraph $\mathcal{G}(t)$ that is uniformly connected over a finite time horizon. The Hilbert metric $V_{B}(x)=d_{\mathcal{K}}(x,\boldsymbol{1})$ provides the Lyapunov function as in the discrete case.

The seminal paper of \cite{Moreau2005} highlights the underlying geometry of consensus algorithms such as the ones considered here, which is that the convex hull of the states $\{x_1, x_2, \dots, x_n\}$ never expands under the consensus update. The Lyapunov function (\ref{consensus2})  is a measure of the diameter of the convex hull. This insight leads to a number of nonlinear  generalizations of consensus theory. For instance, the linear update can be replaced by an arbitrary monotone update, without altering the convergence analysis \cite{Moreau2003}.

\section{Invariant differential positivity on Lie groups} \label{Lie}

\subsection{Lie groups and Lie algebras}

Let $G$ be a smooth manifold that also has a group structure. Then $G$ is said to be a \emph{Lie group} if the group multiplication $(\,\cdot\,,\,\cdot\,):G\times G\rightarrow G$, $(g_1,g_2)\rightarrow g_1g_2$ and group inverse operations $(\,\cdot\,)^{-1}:G\rightarrow G$, $g\rightarrow g^{-1}$ are smooth mappings.
Given two Lie groups $G$, $H$, the product $G\times H$ is itself a Lie group with the product manifold structure and group operation $(g_1,h_1)(g_2,h_2)=(g_1g_2,h_1h_2)$. A Lie group $H$ is said to be a \emph{Lie subgroup} of a Lie group $G$ if it is a subgroup of $G$ and an immersed submanifold of $G$.  
For fixed $a\in G$, the left and right translation maps 
$L_a,R_a:G\rightarrow G$ are defined by $L_a(g)=ag$ and $R_a(g)=ga$,
respectively. Denote the group identity element in $G$ by $e$ and consider the map $L_{g}$ which maps $e\in G$ to $g\in G$. The diffeomorphism $L_{g}$ induces a vector space isomorphism $dL_{g}\vert_e:T_eG\rightarrow T_gG$. Thus, one can use the differential $dL_{g}\vert_e$ to move objects in the tangent space $T_eG$ to the tangent space $T_gG$. Of course, one can use right translations in a similar way. 

A vector field $X$ on a Lie group $G$ is said to be \emph{left-invariant} if 
\begin{equation}
X_{g_1g_2}=dL_{g_1}\big\vert_{g_2} X_{g_2}, \quad \forall g_1, g_2 \in G.
\end{equation}
Similarly, X is said to be \emph{right-invariant} if $X_{g_1g_2}=dR_{g_2}\vert_{g_1} X_{g_1}$ for all $g_1, g_2 \in G$. Note that a left-invariant vector field is uniquely determined by its value at the identity element $e$ by the equation $X_g=dL_g\vert_eX_e$, for each $g\in G$. We denote the set of all left-invariant vector fields on a Lie group $G$ by $\mathfrak{g}$. We endow $\mathfrak{g}$ with the Lie bracket operation $[\cdot,\cdot]$ on vector fields defined by $[X,Y]\vert_g(q)=X(Y(q))-Y(X(q))\vert_g$, for all points $g\in G$ and smooth functions $q:G\rightarrow \mathbb{R}$. Note that $[\cdot,\cdot]$ is closed on the set $\mathfrak{g}$, in the sense that the Lie bracket of two left-invariant vector fields is itself a left-invariant vector field. The set $\mathfrak{g}$ endowed with the Lie bracket operation is known as the Lie algebra of $G$. Clearly there exists a linear isomorphism between $\mathfrak{g}$ and $T_e G$, given by $X\mapsto X_e$. 

A \emph{one-parameter subgroup} of a Lie group $G$ is a smooth homomorphism $\varphi:(\mathbb{R},+)\rightarrow G$. That is, a curve $\varphi:\mathbb{R}\rightarrow G$ satisfying $\varphi(s+t)=\varphi(s)\varphi(t)$, $\varphi(0)=e$ and $\varphi(-t)=\varphi(t)^{-1}$, for all $s, t\in\mathbb{R}$.
There exists a one-to-one correspondence between one-parameter subgroups of $G$ and elements of $T_eG$, given by the map $\varphi\mapsto d\varphi\vert_0 1$. Furthermore, for each $X\in\mathfrak{g}$, there exists a unique one parameter subgroup $\varphi_X:\mathbb{R}\rightarrow G$ such that $\varphi'_X(0)=X$.
The \emph{exponential map} $\exp:\mathfrak{g}\rightarrow G$ of a Lie group $G$ is defined by $\exp(X):=\varphi_X(1)$, where $\varphi_X$ is the unique one-parameter subgroup corresponding to $X\in\mathfrak{g}$.
The curve $\gamma(t)=\varphi_{tX}(1)=\varphi_X(t)$ is the unique homomorphism in $G$ with $\gamma'(0)=X$. Note that $\exp:\mathfrak{g}\rightarrow G$ maps a neighborhood of $0\in\mathfrak{g}$ diffeomorphically onto a neighborhood of $e\in G$. 

For a matrix group $G$ arising as a subgroup of the general linear group, the exponential map $\exp:\mathfrak{g}\rightarrow G$ coincides with the usual exponential map for matrices defined by
$
e^A=I+A+\frac{A^2}{2}+\cdot\cdot\cdot = \sum_{n=0}^\infty\frac{1}{n!}A^n
$
for a real $n\times n$ matrix $A$. Moreover, for a matrix group $G$ the Lie bracket on $\mathfrak{g}$ coincides with the usual Lie bracket operation on matrices given by $[A,B]=AB-BA$, and the differential of the left translation map is given by
$
dL_g\vert_eX=gX$, for $X\in\mathfrak{g}$.
It is also well-known that for matrix groups the \emph{adjoint} representation takes the form
$\mathrm{Ad}(g)X=gXg^{-1}$, $\forall g\in G, \forall X\in\mathfrak{g}$.

A Finsler structure $F:TG\rightarrow [0,\infty)$ on a Lie group $G$ is said to be left-invariant if 
\begin{equation}
F(e,\xi)=F(g,dL_g\vert_e\xi), \quad \forall g\in G, \; \forall \xi \in \mathfrak{g}.
\end{equation}
That is, a Finsler metric is left-invariant if the left translations are isometries on $G$. Similarly, 
$F$ is said to be right-invariant if $F(e,\xi)=F(g,dR_g\vert_e\xi)$ for all $g\in G$ and $\xi\in \mathfrak{g}$. A Finsler metric that is both left-invariant and right-invariant is called bi-invariant. Invariant Riemannian metrics on $G$ are defined in a similar way.

\subsection{Cartan connections}

Here we review the Cartan affine connections \cite{Postnikov2013} associated with a Lie group $G$, which will play an important role in deriving linearizations of continuous-time systems defined on Lie groups for studying invariant differential positivity. First, recall that an \emph{affine connection} $\nabla$ on a manifold $\mathcal{M}$ is a bilinear map $\nabla:C^{\infty}(\mathcal{M},T\mathcal{M})\times C^{\infty}(\mathcal{M},T\mathcal{M})\rightarrow C^{\infty}(\mathcal{M},T\mathcal{M})$, $(X,Y)\mapsto\nabla_X Y$, where $C^{\infty}(\mathcal{M},T\mathcal{M})$ denotes the set of smooth vector fields on $\mathcal{M}$, such that for every smooth function $f:\mathcal{M}\rightarrow \mathbb{R}$ and smooth vector fields $X,Y$, (i) $\nabla_{fX}Y=f\nabla_X Y$ and (ii) $\nabla_X(fY)=X(f)Y+f\nabla_X Y$. An affine connection gives a well-defined notion of a covariant derivative $\nabla_\xi X$ of a vector field $X$ with respect to a vector $\xi$ at any point on $\mathcal{M}$. An affine connection $\nabla$ also determines affine geodesics $\gamma=\gamma(t)$, which are defined as curves that satisfy
\begin{equation} \label{affine geo}
\nabla_{\gamma'(t)}\gamma'(t)=0, 
\end{equation}
at all points along the curve. An affine connection defines the \emph{torsion} $T$ and \emph{curvature} $R$ tensors by the formulae
\begin{equation} \label{torsion def}
T(X,Y)=\nabla_X Y -\nabla_Y X - [X,Y],
\end{equation}
\begin{equation}
R(X,Y)Z = \nabla_X \nabla_Y Z - \nabla_Y \nabla_X Z - \nabla_{[X,Y]}Z,
\end{equation}
for $X,Y,Z\in C^{\infty}(\mathcal{M},T\mathcal{M})$.

\begin{definition}
A connection $\nabla$ on a Lie group $G$ is said to be left-invariant if for any two left-invariant vector fields $X$ and $Y$, $\nabla_X Y$ is also left-invariant. That is,
\begin{equation}
X,Y\in \mathfrak{g} \implies \nabla_X Y\in\mathfrak{g}.
\end{equation}
\end{definition}
A left-invariant connection is characterized by an $\mathbb{R}$-bilinear mapping $\alpha:\mathfrak{g}\times\mathfrak{g}\rightarrow\mathfrak{g}$ given by $\alpha(X,Y)=\nabla_X Y$, which provides a bijective correspondence between left-invariant connections on $G$ and multiplications on  $\mathfrak{g}$. Each multiplication $\alpha$ admits a unique decomposition $\alpha = \alpha' +\alpha''$ into a symmetric part $\alpha'$ and a skew-symmetric part $\alpha''$.

\begin{definition}
A left-invariant connection $\nabla$ on a Lie group $G$ is said to be a Cartan connection if the affine geodesics through the identity element coincide with the one-parameter subgroups of $G$.
\end{definition}

A left-invariant connection $\nabla$ is a Cartan connection if and only if the corresponding map $\alpha$ satisfies $\alpha(X,X)=0$ for all $X\in\mathfrak{g}$; i.e. iff it is skew-symmetric \cite{Postnikov2013}. The one-dimensional family of connections characterized by 
\begin{equation}
\alpha(X,Y)=\lambda[X,Y], \quad X,Y\in\mathfrak{g},
\end{equation}
where $\lambda\in\mathbb{R}$ generate Cartan connections. The corresponding torsion $T$ and curvature $R$ tensors take the form
\begin{equation}
T(X,Y)=(2\lambda-1)[X,Y], \quad X,Y\in\mathfrak{g}, 
\end{equation}
\begin{equation}
R(X,Y)Z=(\lambda^2-\lambda)[[X,Y],Z],\quad X,Y,Z\in\mathfrak{g}.
\end{equation}

For $\lambda = 0$ and $\lambda =1$, we obtain the two canonical Cartan connections whose curvature tensors vanish. The connection corresponding to $\lambda = 0$ is called the \emph{left Cartan connection} and satisfies
\begin{equation}
\nabla_X Y = 0, \quad T(X,Y) = -[X,Y], \quad \forall X,Y\in\mathfrak{g}.
\end{equation}
This connection is the unique connection on $G$ with respect to which the left-invariant vector fields are covariantly constant. That is, for any $\xi\in T_gG$ and $X\in\mathfrak{g}$, we have $\nabla_{\xi}X = 0$. Indeed a vector field is covariantly constant with respect to the left Cartan connection if and only if it is left-invariant.

The connection corresponding to $\lambda = 1$ is called the \emph{right Cartan connection} and satisfies
\begin{equation}
\nabla_X Y = [X,Y], \quad T(X,Y) = [X,Y], \quad \forall X,Y\in\mathfrak{g}.
\end{equation}
In analogy with the left Cartan connection, a vector field is covariantly constant with respect to the right Cartan connection if and only if it is right-invariant. Finally, we note that for the left and right Cartan connections, the torsion tensor $T$ is covariantly constant:
\begin{equation}
\nabla T = 0.
\end{equation}

\subsection{Invariant cone fields on Lie groups}

Let $G$ be a Lie group with Lie algebra $\mathfrak{g}$.  A cone field $\mathcal{C}_G$ is \emph{left-invariant} if it is invariant with respect to a left-invariant frame. This is equivalent to the following definition.

\begin{definition}
A cone field $\mathcal{C}_G$ on a Lie group $G$ is said to be \emph{left-invariant} if 
\begin{equation}
\mathcal{C}_G(g_1g_2)=dL_{g_1}\big\vert_{g_2}\mathcal{C}_G(g_2),
\end{equation}
for all $g_1,g_2\in G$.
\end{definition}

Note that a left-invariant cone field is characterized by the cone in the tangent space at identity $T_eG=\mathfrak{g}$. That is, given a cone $\mathcal{C}$ in $\mathfrak{g}$, the corresponding left-invariant cone field is given by
$\mathcal{C}_G(g)=dL_g\vert_e\mathcal{C}$, 
for all $g \in G$. 
For example, if we are given a polyhedral cone $\mathcal{C}$ in $T_eG$ that is specified via a collection of inequalities
$
x\in T_eG: \langle n_i,x\rangle_e\geq  0$, 
where $\{n_i\}$ is a collection of vectors in $T_eG$ and $\langle\cdot,\cdot\rangle_e$ is an inner product on $T_eG$, then the corresponding left-invariant cone field of rank $k$ can be defined by the collection of inequalities  
$\delta g\in T_g G:\langle dL_g\big\vert_e n_i,\delta g\rangle_g\geq  0$,
for all $g\in G$, where $\langle\cdot,\cdot\rangle_g$ is the unique left-invariant Riemannian metric corresponding to the inner-product $\langle\cdot,\cdot\rangle_e$  in $T_eG$.

Similarly, given a quadratic cone $\mathcal{C}$ of rank $k$ in $T_eG$ defined by $\langle x, Px \rangle_e\geq 0$ for $x\in T_eG$, where $P$ is a symmetric invertible $n\times n$ matrix with $k$ positive eigenvalues and $n-k$ negative eigenvalues, the corresponding left-invariant cone field is given by
$\delta g\in T_gG:\langle \delta g,dL_g\vert_e P dL_g^{-1}\vert_g\delta g\rangle_g\geq  0$, 
for all $g\in G$, where once again $\langle\cdot,\cdot\rangle_g$ denotes the left-invariant Riemannian metric on $G$ corresponding to the inner-product $\langle\cdot,\cdot\rangle_e$. 
A graphical representation of a left-invariant pointed convex cone field is shown in Figure~\ref{fig:LeftCone}.
Of course, analogously one can define notions of right-invariant cone fields using the vector space isomorphisms $dR_g\vert_{e}$ induced by right translations on $G$. We call a cone field \emph{bi-invariant} if it is both left-invariant and right-invariant.

\begin{figure}
\centering
\includegraphics[width=0.45\linewidth]{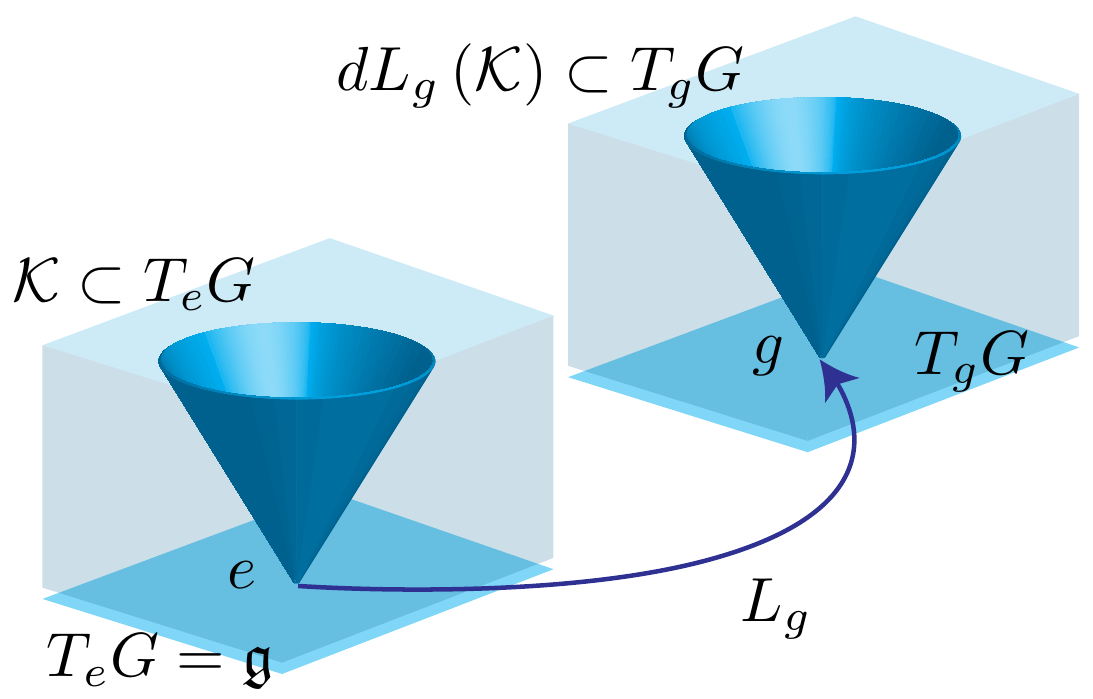}
  \caption{A left-invariant cone field $\mathcal{K}_G$ on a Lie group $G$. The cone at point $g\in G$ is given by $\mathcal{K}_G(g)=dL_g\vert_e(\mathcal{K})$, where $\mathcal{K}$ denotes the cone at the identity element $e\in G$.
  }
  \label{fig:LeftCone}
\end{figure}

\subsection{Invariant differential positivity}

Let $G$ be a Lie group and consider the discrete-time dynamical system defined by
$
g^+=F(g)$, 
where $F:G\rightarrow G$ is smooth. From a differential analytic point of view, we are interested in properties of the system that can be derived by studying its linearization, which can be expressed as
\begin{equation}  \label{discrete2}
\delta g^+=dF\big\vert_g\delta g, \quad \forall g\in G.
\end{equation}
Now the system is differentially positive with respect to a left-invariant cone field $\mathcal{C}(g):=dL_g\vert_e\mathcal{C}$ if
$
d F\vert_g\mathcal{C}(g)\subseteq \mathcal{C}(F(g))=d L_{F(g)}\vert_e\mathcal{C}$,
where $\mathcal{C}$ denotes the cone at the identity element $e$. We can reformulate this as a condition expressed on a single tangent space $T_g G$:
$
dL_{g F(g)^{-1}}\vert_{F(g)}\circ d F\vert_g\mathcal{C}(g)\subseteq \mathcal{C}(g)$.
The virtue of this reformulation is that it provides a characterization of differential positivity in the form of a pointwise positivity condition on a linear map
\begin{equation}
 \hat{A}(g):=dL_{g F(g)^{-1}}\big\vert_{F(g)}\circ d F\big\vert_g:T_gG\rightarrow T_gG
 \end{equation}
defined on each tangent space. Indeed, we can go further and identify each tangent vector in $T_gG$ with an element of $T_eG$ through left translation and thus a vector $\mathbf{v}\in\mathbb{R}^n$ via the \emph{vectorization} or $^{\vee}$ map $\left(\,\cdot\,\right)^{\vee}:T_e G\rightarrow \mathbb{R}^n$, where $n=\operatorname{dim}G$. Differential positivity of $g^+=F(g)$ with respect to a left-invariant cone field generated by a cone $\mathcal{C}$ in $\mathbb{R}^n$ identified with $T_eG$ is now reduced to positivity of the linear map
 \begin{equation}
\mathbf{v}^+=A(g)\mathbf{v}, \quad \mathbf{v}\in\mathbb{R}^n
\end{equation}
with respect to $\mathcal{C}$ for all $g\in G$, where $A(g):\mathbb{R}^n\rightarrow\mathbb{R}^n$ denotes the linear map
corresponding to $\hat{A}(g)$ 
after $T_gG$ is identified with $\mathbb{R}^n$ as described. See Figure~\ref{fig:invdiffpos} for a visual representation of these ideas in the context of differential positivity with respect to closed, convex, and pointed cones.

\begin{figure}
\centering
\includegraphics[width=0.8\linewidth]{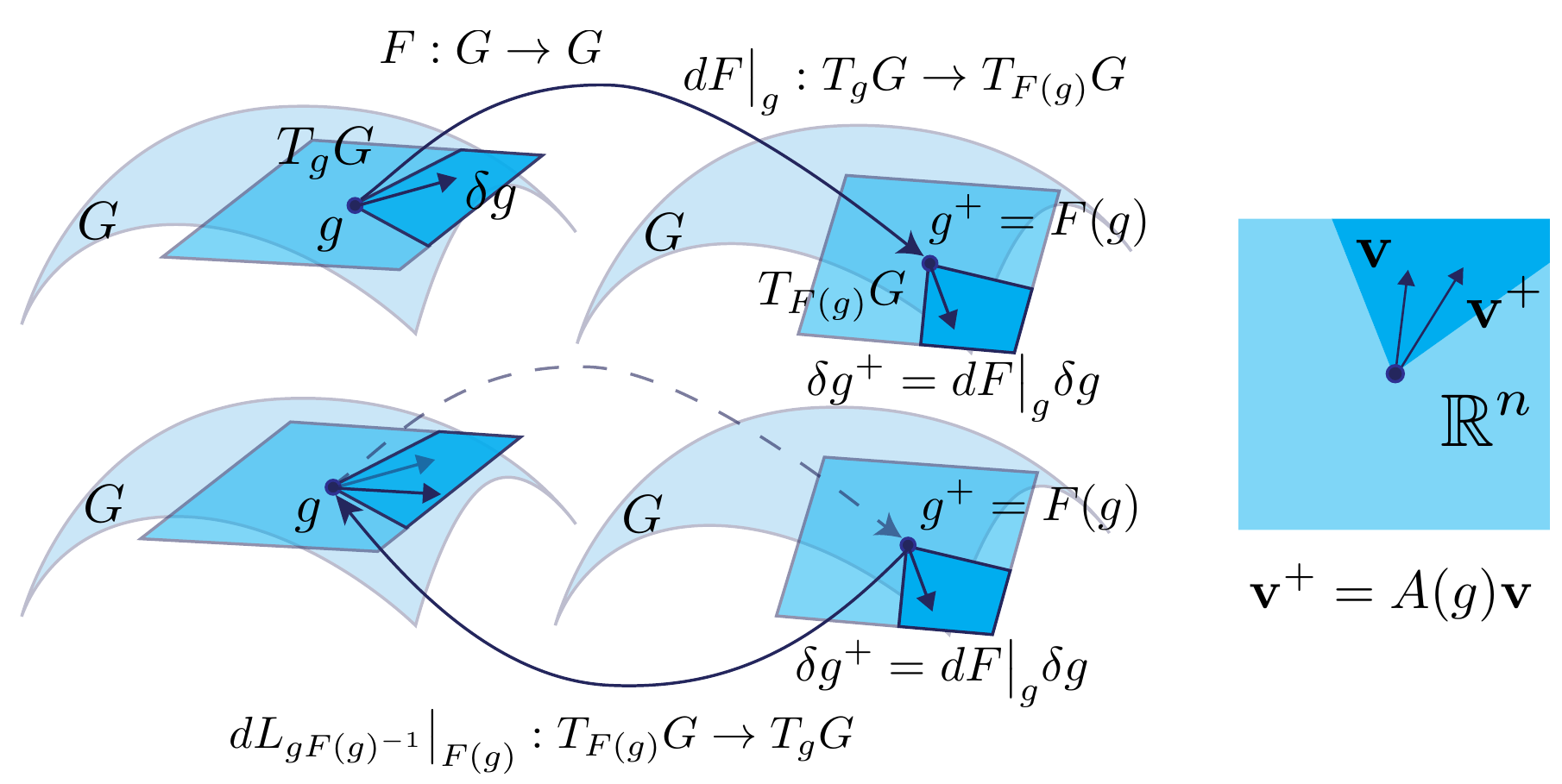}
  \caption{Invariant differential positivity of a discrete-time dynamical system $g^+=F(g)$ on a Lie group $G$, with respect to a left-invariant cone field. Invariant differential positivity can be reformulated as the positivity of a collection of linear maps $\{A(g):\mathbb{R}^n\rightarrow\mathbb{R}^n\mid g\in G\}$ with respect to a constant cone $\mathcal{K}$ in $\mathbb{R}^n$.
  }
  \label{fig:invdiffpos}
\end{figure}

Now consider a continuous-time dynamical system $\Sigma$ on $G$ given by 
$
\dot{g}= f(g)$,
where $f$ is a smooth vector field that assigns a vector $f(g):=X_g\in T_g G$ to each point $g\in G$. 
Let $\psi_t(g)$ denote the trajectory of $\Sigma$ at time $t\in\mathbb{R}$ with initial point $g\in G$. Note that the flow $\psi_t:G\rightarrow G$ is a diffeomorphism with differential $d\psi_t\vert_g:T_gG\rightarrow T_{\psi_t(g)}G$. The linearization of $\Sigma$ with respect to a left-invariant frame on $G$ takes the form
\begin{equation} \label{6.3}
\frac{d}{dt}\delta g=\lim_{t\rightarrow 0}\frac{dL_{g \psi_t(g)^{-1}}\vert_{\psi_t(g)}\circ d\psi_t\vert_g\,\delta g -\delta g}{t}.
\end{equation}
To see this, note that the tangent vector $\delta g\in T_gG$ evolves under the flow to $d\psi_t\vert_g\delta g\in T_{\psi_t(g)}G$, which is then pulled back to $T_gG$ using $dL_{g \psi_t(g)^{-1}}\vert_{\psi_t(g)}$ by left-invariance, in order to compute the derivative $\dot{\delta g}$. Let $\{E_i\}$ denote a left-invariant frame consisting of a collection of left-invariant vector fields on $G$ such that $\{E_i\vert_g\}$ forms a basis of $T_gG$. For each $t>0$, we can write $\delta g (t) := d\psi_t\vert_g\delta g = \sum_i(\delta g(t))^i E_i\vert_{\psi_t(g)}$, for some $(\delta g(t))^i\in\mathbb{R}$. The expression in (\ref{6.3}) becomes
\begin{align} 
\frac{d}{dt}\delta g&=\lim_{t\rightarrow 0}\frac{dL_{g \psi_t(g)^{-1}}\vert_{\psi_t(g)}\sum_i(\delta g(t))^iE_i\vert_{\psi_t(g)}  -\sum_i(\delta g(0))^iE_i\vert_{g}}{t} \\
&=\lim_{t\rightarrow 0}\left(\frac{\sum_i(\delta g(t))^i  -\sum_i(\delta g(0))^i}{t}\right)E_i\vert_{g}. \label{components Cartan 1}
\end{align}

\begin{remark}
If the vector field $f(g)=X_g$ is left-invariant so that 
$
X_g=dL_g\vert_eX_e, 
$
where $X_e\in T_eG$ is the vector at identity, then the solution $\psi_t(e)$ of $\dot{g}=X_g$ through the identity element $e$ is given by
$
\psi_t(e)=\exp(tX_e)
$
and the flow $\psi_t$ now takes the form
$
\psi_t(g)=L_g\psi_t(e)= L_g\exp(tX_e)$,
by the left-invariance of $X_g$. The linearized dynamics reduces to
\begin{equation}
\frac{d}{dt}\delta g =\lim_{t\rightarrow 0}\frac{dL_{g \psi_t(g)^{-1}}\vert_{\psi_t(g)}\circ dL_{\psi_t(g) g^{-1}}\vert_g\delta g-\delta g}{t}=0.
\end{equation}
Assuming that the Lie group is endowed with a left-invariant cone field $\mathcal{C}(g):=dL_g\vert_e\mathcal{C}$, we see that the system is differentially positive when the flow is given by a left-invariant vector field. Specifically, we have
$
d\psi_t\vert_g\mathcal{C}(g)=\mathcal{C}\left(\psi_t(g)\right)$,
for all $t\in\mathbb{R}$ and $g\in G$. Note that the differential positivity of the system is not strict. The differential positivity of a left-invariant vector field with respect to a left-invariant cone field on a Lie group is the Lie group analogue of the trivial case of differential positivity of a uniform vector field in $\mathbb{R}^n$ 
with respect to a constant cone field.
\end{remark}

The linearization with respect to a left-invariant frame given in (\ref{6.3}) coincides precisely with the linearization obtained through covariant differentiation using the left Cartan connection on $G$. 
To see this, not that given a tangent vector $\delta g\in T_g G$, the covariant derivative
\begin{equation} \label{left covariant}
\left(\nabla_{\delta g} X\right)\big\vert_g\in T_g G,
\end{equation}
is a measure of the change in the vector field $X$ in the direction of $\delta g$ at $g\in G$. Note that we can express the vector field $X$ with respect to a left-invariant frame $\{E_i\}$ as $X_g=\sum_i X^i(g)E_i\vert_g$, where $X^i:G\rightarrow \mathbb{R}$ are smooth functions. By the properties of affine connections, we have
\begin{align}
\nabla_{\delta g}X\vert_g&= \nabla_{\delta g}\left(\sum_i X^i(g)E_i\vert_g\right) =\sum_i \left(\delta g(X^i)(g)E_i\vert_g +   X^i(g)\nabla_{\delta g}E_i\vert_g\right) \\
&= \sum_i(dX^i\vert_g\delta g)E_i\vert_g,  \label{components Cartan 2}
\end{align}
where we have used the fact that $\nabla_{\delta g}E_i = 0$ as  $\nabla$ is the left Cartan connection and $E_i$ is a left-invariant vector field. Clearly the components in (\ref{components Cartan 1}) and (\ref{components Cartan 2}) agree as $\psi_t$ is the flow defined by the vector field $X$.

Now (\ref{left covariant}) defines a linear operator in $\delta g$ at any fixed $g\in G$, which we denote by $\hat{A}(g):T_gG\rightarrow T_gG$. Furthermore, through identification of $\mathfrak{g}=T_eG$ and $\mathbb{R}^n$ via the vectorization map $^{\vee}$, we can equivalently consider a collection of linear maps $A(g):\mathbb{R}^n\rightarrow \mathbb{R}^n$ and note that 
differential positivity with respect to a left-invariant cone field generated by a cone  $\mathcal{C}\subset\mathbb{R}^n$ reduces to the positivity of the linear map
\begin{equation}
\dot{\mathbf{v}}=A(g)\mathbf{v}, \quad \mathbf{v}\in\mathbb{R}^n
\end{equation}
with respect to $\mathcal{C}$ for all $g\in G$. 
This choice of covariant differentiation precisely yields the notion of linearization that we need for the study of invariant differential positivity with respect to a left-invariant cone field, since it corresponds to the unique connection with respect to which all left-invariant vector fields are covariantly constant. In particular, note that for a left-invariant vector field $X$, the linearization $\hat{A}(g)$ is immediately seen to be null since $\hat{A}(g)\delta g = \nabla_{\delta g} X = 0$. Analogously, the right Cartan connection can be used to derive the appropriate linearization for the study of differential positivity with respect to right-invariant cone fields.

A key contribution of \cite{Forni2015} is the generalization of Perron-Frobenius theory of a linear system that is positive with respect to a closed, convex, and pointed cone 
 to nonlinear systems within a differential framework, whereby the the Perron-Frobenius eigenvector of linear positivity theory is replaced by a Perron-Frobenius vector field $w(x)$ whose integral curves shape the attractors of the system. The main result on closed differentially positive systems is that the asymptotic behavior is either captured by a Perron-Frobenius curve $\gamma$ such that 
$
\gamma'(s)=w(\gamma(s))
$
at every point on $\gamma$; or is the union of the limit points of a trajectory that is nowhere aligned with the Perron-Frobenius vector field, which is a highly non-generic situation.

In Section \ref{circle}, we will make use of the following theorem to study consensus on the circle. The theorem is a special case of Theorem 5 of  \cite{Forni2015a} applied to invariant cone fields on Lie groups. 

 \begin{theorem} \label{thm compact}
 Let $\Sigma$ be a uniformly strictly differentially positive system with respect to a left-invariant cone field $\mathcal{K}$ of closed, convex, and pointed cones in a bounded, connected, and forward invariant region $S\subseteq G$ of a Lie group $G$ equipped with a left-invariant Finsler metric. If the normalized Perron-Frobenius vector field $w$ is left-invariant, i.e. $w(g)=dL_g\vert_e w(e)$, and satisfies
$\limsup_{t\rightarrow\infty}\|d\psi_t\vert_g w(g)\|_{\psi_t(g)}<\infty$,
then there exists a unique integral curve of $w$ that is an attractor for all the trajectories of $\Sigma$ from $S$. Moreover, the attractor is a left translation of a subgroup of $G$ that is isomorphic to $\mathbb{S}^1$.
 \end{theorem}

It should be noted that strict differential positivity with respect to a left-invariant cone field does not necessarily imply the existence of a left-invariant Perron-Frobenius vector field. However, such a special case is particularly tractable for analysis and results in elegantly simple attractors which prove sufficient for the analysis of the consensus problems we study in this paper. The following example from \cite{Forni2014a} concerns a system that is differentially positive with respect to a left-invariant cone field on a Lie group, whose Perron-Frobenius vector field is not left-invariant.

\subsection{Example: nonlinear pendulum}

Here we briefly review the differential positivity of the classical nonlinear planar pendulum equation which takes the form of a flow on the cylinder $G=\mathbb{S}^1\times\mathbb{R}$. Thinking of $\mathbb{S}^1$ as being embedded in the complex plane, we can represent elements of the Lie group $G$ by $(e^{i\theta},v)$. The pendulum equation can be written in the form $\dot{g}=dL_g\vert_e\Omega(g)$, where $\Omega:G\rightarrow\mathfrak{g}$ is specified by $\Omega^{\vee}=(\Omega_1,\Omega_2)^T$, where $\Omega_1$ is a purely imaginary number and $\Omega_2$ is real. We have

\begin{equation}
{d\over dt}
\begin{pmatrix}
e^{i\theta} \\
v
\end{pmatrix}
= \begin{pmatrix}
e^{i\theta} & 0 \\
0 & 1
\end{pmatrix}
\begin{pmatrix}
\Omega_1 \\
\Omega_2 
\end{pmatrix}
\end{equation}
for $\Omega_1=iv$, 
$\Omega_2 = -\sin\theta-\rho v+u$,
where $\rho\geq 0$ is the damping coefficient, and $u$ is a constant torque input. Thus, the linearized dynamics at point $(e^{i\theta},v)$ is governed by the linear map $A(g):\mathbb{R}^2\rightarrow\mathbb{R}^2$ given by
\begin{equation}
A(g)
=\begin{pmatrix} 0  & 1 \\
-\cos\theta & -\rho 
\end{pmatrix}
\end{equation}
It is easy to verify that the map $A(g):\mathbb{R}^2\rightarrow\mathbb{R}^2$ is strictly positive with respect to the cone
$\mathcal{K}:=\{(\mathrm{v}_1,\mathrm{v}_2)\in \mathbb{R}^2:\mathrm{v}_1\geq 0,\; \mathrm{v}_1+\mathrm{v}_2 \geq 0\}$,
for $\rho > 2$ by showing that at any point on the boundary of the cone $\mathcal{K}$, the vector $A(g)(\mathrm{v}_1,\mathrm{v}_2)^T$ is oriented towards the interior of the cone for any $g\in G$. It can also be shown that every trajectory belongs to a forward invariant set $S\subseteq G$ such that $(\Omega_1,\Omega_2)\in\operatorname{int}\mathcal{K}$ after a finite amount of time. Differential positivity can now be used to establish the existence of a unique attractive limit cycle in $S$. See \cite{Forni2014a} for details. Thus we see that the nonlinear pendulum model defines a flow that is monotone in the sense of invariant differential positivity. 

\begin{remark}
An important difference between transversal contraction theory \cite{Manchester2014, Forni2014Lyap} and differential positivity is that in transversal contraction theory the dominant distribution must be known precisely everywhere in order to define the necessary contraction metric, whereas in differential positivity theory it is not necessary to have a precise knowledge of the dominant or Perron-Frobenius distribution to be able to conclude the existence of an attractor. A clear illustration of this is provided by the nonlinear pendulum model discussed above, which is shown to have a limit cycle through strict differential positivity, without any precise knowledge of the dominant distribution itself. 
\end{remark}

\subsection{Invariant differential positivity with respect to cones of rank $k$}

By replacing the classical Perron-Frobenius theorem for systems that are positive with respect to closed, convex, and pointed cones with the generalization provided by Theorem \ref{PF k},
and the notion of strict differential positivity with respect to a closed, convex, and pointed cone field with that of uniform strict differential positivity with respect to a cone field of rank $k$, we arrive at a generalization of differential Perron-Frobenius theory whereby the attractors of the system are shaped not by a Perron-Frobenius vector field, but by a smooth distribution of rank $k$. See Figure~\ref{fig:highrank}.

  \begin{figure}
\centering
\includegraphics[width=0.5\linewidth]{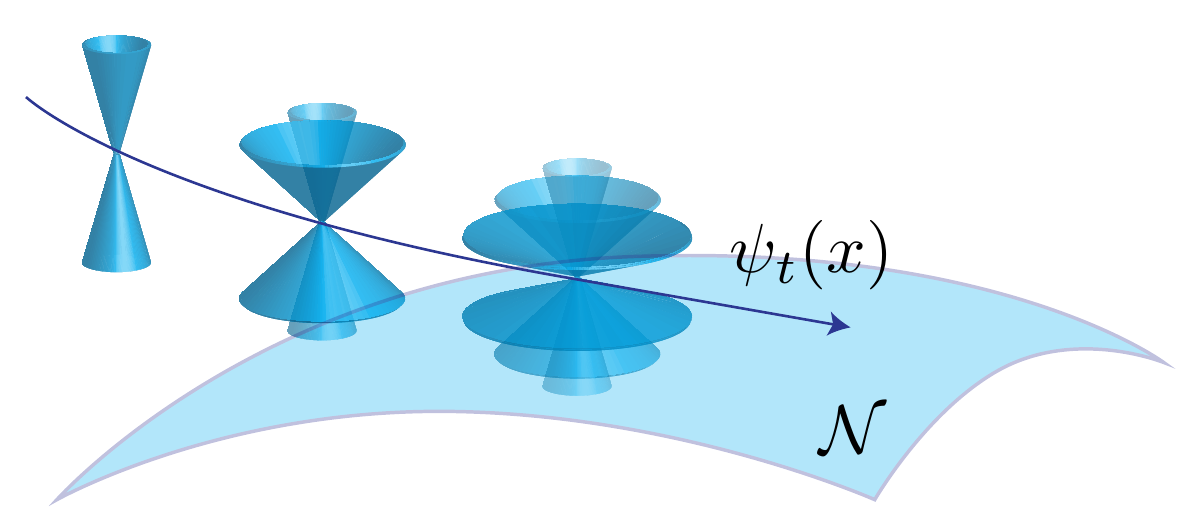}
  \caption{Asymptotic convergence to a two-dimensional integral submanifold $\mathcal{N}$ for a dynamical system that is strictly differentially positive with respect to a cone field of rank $2$.
  }
  \label{fig:highrank}
\end{figure}

Recall that a smooth distribution $\mathcal{D}$ of rank $k$ on a smooth manifold $\mathcal{M}$ is a rank-$k$ smooth subbundle of $T\mathcal{M}$. A rank $k$ distribution is often described by specifying a $k$-dimensional linear subspace $\mathcal{D}_x\subseteq T_x\mathcal{M}$ at each point $x\in\mathcal{M}$, and writing $\mathcal{D}=\cup_{x\in\mathcal{M}}\mathcal{D}_x$. It follows from the local frame criterion for subbundles that $\mathcal{D}$ is a smooth distribution if and only if each point $x\in\mathcal{M}$ has a neighborhood $\mathcal{U}$ on which there are smooth vector fields $X_1,..., X_k$ such that $\{X_j\vert_{\tilde{x}}:j=1,...,k\}$ forms a basis for $\mathcal{D}_{\tilde{x}}$ at each point $\tilde{x}\in\mathcal{U}$ \cite{Lee2003}. The distribution $\mathcal{D}$ is then said to be locally spanned by the vector fields $X_j$.

Given a smooth distribution $\mathcal{D}\subseteq T\mathcal{M}$, a nonempty immersed submanifold $\mathcal{N}\subseteq\mathcal{M}$ is said to be an \emph{integral manifold} of $\mathcal{D}$ if 
\begin{equation}
 T_x\mathcal{N}=\mathcal{D}_x \quad \forall x\in\mathcal{N}.
\end{equation}
The question of whether for a given distribution there exists an integral manifold is intimately connected to the notion of \emph{involutivity} and characterized by the \emph{Frobenius theorem}. A distribution $\mathcal{D}$ is said to be involutive if given any pair of smooth vector fields $X,Y$ defined on $\mathcal{M}$ such that $X_x,Y_x\in\mathcal{D}_x$ for each $x\in\mathcal{M}$, the Lie bracket $[X,Y]\vert_x$ also lies in $\mathcal{D}_x$. By the local frame criterion for involutivity, one can show that a distribution $\mathcal{D}$ is involutive if
there exists a smooth local frame $\{X_j:j=1,...,k\}$ for $\mathcal{D}$ in a neighborhood of every point in $\mathcal{M}$ such that $[X_i,X_j]$ is a section of $\mathcal{D}$ for each $i,j=1,...,k$. The Frobenius theorem tells us that involutivity of a distribution is a necessary and sufficient condition for the existence of an integral manifold through every point \cite{Lee2003}. The following result is a generalization of Theorem \ref{thm compact} to systems that are invariantly differentially positive with respect to higher rank cone fields on Lie groups.

 \begin{theorem} \label{thm compact2}
 Let $\Sigma$ be a uniformly strictly differentially positive system with respect to a left-invariant cone field $\mathcal{C}$ of rank $k$ in a bounded, connected, and forward invariant region $S\subseteq G$ of a Lie group $G$ equipped with a left-invariant Finsler metric. If there exists an involutive
distribution $\mathcal{D}$ satisfying $\mathcal{D}_g\in\operatorname{int}\mathcal{C}(g)\setminus\{0_g\}$ such that for every $w(g)\in\mathcal{D}_g$
 \begin{equation}
\limsup_{t\rightarrow\infty}\|d\psi_t\vert_g w(g)\|_{\psi_t(g)}<\infty,
 \end{equation}
 and for all $g\in S$ and $t\geq 0$:
 \begin{equation}
 \mathcal{D}_{\psi_t(g)} = d\psi_t\vert_g \mathcal{D}_g,
 \end{equation}
 then there exists an integral manifold $\mathcal{N}$ of $\mathcal{D}$ that is an attractor for all the trajectories of $\Sigma$ from $S$.
 \end{theorem}
 
\begin{proof}
For any $g\in G$, $t>0$, define the linear map $\Gamma_{g,t}:T_gG\rightarrow T_gG$ by
\begin{equation}
\Gamma_{g,t}(\delta g)=\left(dL_{g\psi_t(g)^{-1}}\big\vert_{\psi_t(g)}\circ d\psi_t\big\vert_g\right)(\delta g).
\end{equation}
By strict differential positivity with respect to the left-invariant cone field $\mathcal{C}$ of rank $k$, there exist unique $\Gamma_{g,t}$-invariant  subspaces  $\mathcal{W}_1^{g,t}$ and $\mathcal{W}_2^{g,t}$ of $T_g G$ such that $\dim \mathcal{W}_1^{g,t} = k$, $\dim \mathcal{W}_2^{g,t} = n-k$, $T_g G=\mathcal{W}_1^{g,t} \oplus \mathcal{W}_2^{g,t}$, and $\mathcal{W}_1^{g,t}\subset \operatorname{int} \mathcal{C}(g)$, $\mathcal{W}_2^{g,t}\cap\mathcal{C}=\{0_g\}$ according to Theorem \ref{PF k}. Moreover, $|\lambda_1^{g,t}| > |\lambda_j^{g,t}|$ for $j\neq 1$, where $\lambda_1^{g,t}$ denotes a $\mathcal{W}_1^{g,t}$-eigenvalue of $\Gamma_{g,t}$ and $\lambda_j^{g,t}$ denote eigenvalues of $\Gamma_{g,t}$ corresponding to $\mathcal{W}_2^{g,t}$ for $j\neq 1$. For any $\delta g \in T_gG$, we write $\delta g = \delta g_1 + \delta g_2$, where $\delta g_1\in \mathcal{W}_1^{g,t}$, $\delta g_2\in\mathcal{W}_2^{g,t}$. Define the map $\Phi_g:\mathcal{C}(g)\setminus\{0_g\}\rightarrow \mathbb{R}^{\geq0}$ by 
\begin{equation}
\Phi_g(\delta g)=\frac{\|\delta g_2\|_g}{\|\delta g_1\|_g}.
\end{equation}
It is clear that this map is well-defined since $\delta g_1\neq 0_g$ for $\delta g \in \mathcal{C}(g)\setminus\{0_g\}$. We have
\begin{equation}
\Phi_g\left(\Gamma_{g,t}\delta g\right) = \frac{\|\Gamma_{g,t}\delta g_2\|_g}{\|\Gamma_{g,t}\delta g_1\|_g} \leq \frac{\max_{j\neq 1}|\lambda_j^{g,t}|}{|\lambda_1^{g,t}|}\frac{\|\delta g_2\|_g}{\|\delta g_1\|_g} < \Phi_g(\delta g).
\end{equation}
Moreover, since $\Sigma$ is assumed to be uniformly strictly differentially positive, there exists $T>0$ and $\nu\in(0,1)$, such that
$\Phi_g\left(\Gamma_{g,t}\delta g \right) \leq \nu \Phi_g(\delta g)$ for all $ t\geq T$. Thus, we have 
\begin{equation}
\Phi_g\left(\Gamma_{g,t}\delta g\right) \leq \nu^n \Phi_g(\delta g), \quad \forall t\geq nT.
\end{equation}
Letting $n\rightarrow \infty$, we see that $\Phi_g\left(\Gamma_{g,t}\delta g\right)\rightarrow 0$ as $t \rightarrow \infty$ for any $\delta g \in \mathcal{C}(g)\setminus \{0_g\}$.

Now if $w(g) \in \mathcal{D}_g$, then  $\limsup_{t\rightarrow\infty}\|d\psi_t\vert_g w(g)\|_{\psi_t(g)}<\infty$ by assumption, which is equivalent to $\limsup_{t\rightarrow\infty}\|\Gamma_{g,t} w(g)\|_g < \infty$ by invariance of the Finsler metric. In particular, for any $\delta g\in \mathcal{C}(g)\setminus\{0_g\}$, we have $\limsup_{t\rightarrow\infty}\|\Gamma_{g,t} \delta g_1\|_g < \infty$, and so $\lim_{t\rightarrow\infty}\Phi_g\left(\Gamma_{g,t}\delta g\right)=0$ implies that $\lim_{t\rightarrow\infty}\Gamma_{g,t}\delta g_2=0$ by completeness. If $\delta g\notin \mathcal{C}(g)$, then for some $\alpha >0$ and $w_1\in\mathcal{W}_1^{g,t}$, we have $\delta g + \alpha w_1\in \mathcal{C}(g)\setminus\{0_g\}$ and $\lim_{t\rightarrow\infty}\Phi_g\left(\Gamma_{g,t}\delta g+\alpha w_1\right)=0$, which implies that $\lim_{t\rightarrow\infty}T_{g,t}\delta g_2=0$ once again. Thus, in the limit of $t\rightarrow \infty$, $\delta g(t) = d\psi_t\vert_g\delta g$ becomes parallel to $\mathcal{D}_{\psi_t(g)}$. 

To prove uniqueness of the attractor, assume for contradiction that $\mathcal{N}_1$ and $\mathcal{N}_2$ are two distinct attractive integral manifolds of $\mathcal{D}$ and let $g_1\in \mathcal{N}_1$, $g_2\in\mathcal{N}_2$. By connectedness of $S$, there exists a smooth curve $\gamma$ in S connecting $g_1$ and $g_2$. Since the curve $\psi_t(\gamma(s))$ converges to an integral manifold of $\mathcal{D}$, $\mathcal{N}_1$ and $\mathcal{N}_2$ must be subsets of the same integral manifold of $\mathcal{D}$, which provides the contradiction that completes the proof.
\end{proof}

A special case of interest arises when the distribution $\mathcal{D}$ of dominant modes of the system is itself a left-invariant distribution, i.e. $\mathcal{D}_g=dL_g\vert_e\mathcal{D}_e$ for all $g\in G$. This property is satisfied for the consensus examples covered in sections \ref{circle} and \ref{SO(3)}. In such cases, the distribution is involutive precisely if $\mathcal{D}_e$ forms a subalgebra $\mathfrak{h}$ of the Lie algebra $\mathfrak{g}$ of $G$. Furthermore, the resulting integral manifold of $\mathcal{D}$ which acts as an attractor in $S$ is a left translation of the unique connected Lie subgroup $H$ of $G$ corresponding to $\mathfrak{h}$. Indeed, Theorem \ref{thm compact} follows as precisely such a special case of Theorem \ref{thm compact2} for $k=1$.

\section{Differential positivity and consensus on the circle}  \label{circle}

Throughout this section, we assume that a cone is closed, convex, and pointed.

\subsection{Differential positivity on the $N$-torus $\mathbb{T}^N$}  \label{diff1}

We consider differential positivity of consensus protocols on the $N$-torus $\mathbb{T}^N$ involving coupling functions $f_{ki}$ associated to edges $(k,i)$. First consider the dynamics
\begin{equation}    \label{asymmetricS1}
\dot{\vartheta}_k=\sum_{i:(k,i)\in\mathcal{E}}f_{ki}(\vartheta_i-\vartheta_k),
\end{equation}
on $\mathbb{T}^N$ for a strongly connected communication graph $(\mathcal{V},\mathcal{E})$, where each $f_{ki}$ is odd, $2\pi$-periodic, continuously differentiable on $(-\pi,\pi)$, and satisfies $f_{ki}(0)=0$, $f'_{ki}(\alpha)>0$ for all $\alpha\in(-\pi,\pi)$. The synchronization manifold is given by
\begin{equation}
\mathcal{M}_{\operatorname{sync}}=\{\vartheta\in\mathbb{S}^1\times\ldots\mathbb{S}^1:\vartheta_1=\ldots=\vartheta_N\},
\end{equation}
where $\vartheta=(\vartheta_1,\ldots,\vartheta_N)^T$. The linearized dynamics is given by $\dot{\delta\vartheta} = A(\vartheta)\delta\vartheta$ with
\begin{equation}
\begin{cases}
A_{kk}(\vartheta)&=-\sum_{i:(k,i)\in\mathcal{E}}f_{ki}'(\vartheta_i-\vartheta_k), \\
A_{ki}(\vartheta)&=f_{ki}'(\vartheta_i-\vartheta_k) \quad \mathrm{if} \quad(k,i)\in\mathcal{E}, \\
A_{ki}(\vartheta)&=0 \quad \mathrm{if} \quad (k,i)\notin \mathcal{E}.
\end{cases}
\end{equation}

This linearization is in effect a linearization with respect to the standard smooth global frame on $\mathbb{T}^N$ defined by the $N$-tuple of vector fields $
\left(\frac{\partial}{\partial\vartheta^1},\cdot\cdot\cdot,\frac{\partial}{\partial\vartheta^N}\right)
$.
Note that this frame is both left-invariant and right-invariant since $\mathbb{T}^N$ is an abelian Lie group. Any tangent vector $\delta \vartheta$ can be expressed as $\delta\vartheta =\sum_{i=1}^N \delta\vartheta^i \frac{\partial}{\partial\vartheta^i}\vert_{\vartheta}$ with respect to this frame.  We also associate to each tangent space the standard Euclidean inner product with respect to this frame, which equips $\mathbb{T}^N$ with a bi-invariant Riemannian metric.
The identity $A(\vartheta)\boldsymbol{1}=0$, where $\boldsymbol{1}=(1,\ldots,1)^T$ captures the invariance of the synchronization manifold $\mathcal{M}_{\operatorname{sync}}$.
The positive orthant with respect to $
\left(\frac{\partial}{\partial\vartheta^1},\cdot\cdot\cdot,\frac{\partial}{\partial\vartheta^N}\right)
$ yields an invariant polyhedral cone field $\mathcal{K}_{\mathbb{T}^N}$ on $\mathbb{T}^N$:
\begin{equation}
\mathcal{K}_{\mathbb{T}^N}(\vartheta):=\{\delta \vartheta\in T_{\vartheta}\mathbb{T}^N: \delta \vartheta^i\geq 0\},
\end{equation}
where $\delta\vartheta =\sum_{i=1}^N \delta\vartheta^i \frac{\partial}{\partial\vartheta^i}\vert_{\vartheta}$. Differential positivity of (\ref{asymmetricS1}) with respect to $\mathcal{K}_{\mathbb{T}^N}$ on $\mathbb{T}^N_{\pi}=\{\vartheta\in\mathbb{T}^N:|\vartheta_k-\vartheta_i|<\pi,\;(i,k)\in\mathcal{E}\}$ is clear by observing that on each face $F_i=\{\delta\vartheta:\delta\vartheta^i=0\}$ of the cone, we have $\dot{\delta\vartheta}^i=(A(\vartheta)\delta\vartheta)^i\geq 0$. Moreover, the differential positivity is strict in the case of a strongly connected graph and the Perron-Frobenius vector field is given by $\boldsymbol{1}_{\vartheta}=\sum_{i=1}^N 1\,\partial/\partial\vartheta^i\vert_{\vartheta}$. Also note that $A(\vartheta)\boldsymbol{1}_{\vartheta}=0$ implies that 
\begin{equation}
d\psi_t\vert_{\vartheta}\boldsymbol{1}_{\vartheta}=\boldsymbol{1}_{\psi_t(\vartheta)},
\end{equation}
where $\psi_t$ denotes the flow of the system. Thus, we have
\begin{equation}
\|d\psi_t\vert_{\vartheta}\boldsymbol{1}_{\vartheta}\|_{\psi_t(\vartheta)}=\|\boldsymbol{1}_{\psi_t(\vartheta)}\|_{\psi_t(\vartheta)} = \|dL_{\psi_t(\vartheta)}\vert_e\boldsymbol{1}_e\|_{\psi_t(\vartheta)} =\|\boldsymbol{1}_e\|_e < \infty, \quad \forall t>0,
\end{equation}
where $e$ is the identity element on the torus. Therefore, the condition 
\begin{equation}
\limsup_{t\rightarrow\infty}\|d\psi_t\vert_g w(g)\|_{\psi_t(g)}<\infty
\end{equation}
 of Theorem \ref{thm compact} is clearly satisfied in this example.

Furthermore, note that differential positivity is retained if we replace (\ref{asymmetricS1}) by the inhomogeneous dynamics
 $\dot{\vartheta}_k=\omega_k + \sum_{i:(k,i)\in\mathcal{E}}f_{ki}(\vartheta_i-\vartheta_k)$,
where $\omega_k\in\mathbb{R}$ is the natural frequency of oscillator $\vartheta_k$, since both systems have the following linearization:
\begin{equation}    \label{asymmetricS1 3}
\dot{\delta\vartheta}_k= \sum_{i:(k,i)\in\mathcal{E}}f'_{ki}(\vartheta_i-\vartheta_k)\delta\vartheta_i- \sum_{i:(k,i)\in\mathcal{E}}f'_{ki}(\vartheta_i-\vartheta_k)\delta\vartheta_k.
\end{equation}
We thus arrive at the following general result.
\begin{theorem}   \label{general thm}
Consider $N$ agents $\vartheta_k\in\mathbb{S}^1$ communicating via a strongly connected graph $\mathcal{G}=(\mathcal{V},\mathcal{E})$ according to
\begin{equation}  \label{protocol}
\dot{\vartheta}_k=\omega_k + \sum_{i:(k,i)\in\mathcal{E}}f_{ki}(\vartheta_i-\vartheta_k),
\end{equation}
where each $f_{ki}$ is an odd, $2\pi$-periodic coupling function that is continuously differentiable on $(-\pi,\pi)$, and satisfies $f_{ki}(0)=0$, $f'_{ki}(\alpha)>0$ for all $\alpha\in(-\pi,\pi)$. Let $S$ be a bounded, connected, and forward invariant region such that
\begin{equation}
S\subset \{\vartheta\in\mathbb{T}^N:|\vartheta_k-\vartheta_i|\neq \pi, \; \forall (k,i)\in\mathcal{E}\}.
\end{equation}
Then, every trajectory from $S$ asymptotically converges to a unique integral curve of 
\begin{equation}  \label{int sync}
\boldsymbol{1}_{\vartheta}=\sum_{i=1}^N 1\frac{\partial}{\partial\vartheta^i}\big\vert_{\vartheta}.
\end{equation}
\end{theorem}

There is a clear generalization of the above theorem to the case of time-varying consensus protocols of the same form. Uniform strict differential positivity in the time-varying case is guaranteed for a strongly connected graph that is uniformly connected over a finite time horizon if there exists some $\delta >0$ such that $f'_{ki}(\alpha,t)\geq \delta >0$ for all $\alpha\in(-\pi,\pi)$ at all times $t>0$.

\subsection*{Example: frequency synchronization}

Consider $N$ agents evolving on $\mathbb{S}^1$ and interacting via a connected bidirectional graph according to (\ref{protocol}), where in addition to the assumptions of Theorem \ref{general thm}, we assume that $f_{ki}(\alpha) \rightarrow + \infty$ as $\alpha \rightarrow \pi$ for each of the coupling functions $f_{ki}$, as in the example depicted in Figure~\ref{fig:couplings} $(a)$. Note that $f_{ki}$ and $f_{ik}$ need not be the same function. In this model, agents attract each other with a strength that is monotonically increasing on $(0,\pi)$ and grows infinitely strong as the separation between connected agents approaches $\pi$, thereby ensuring that the dynamics defines a forward invariant flow on the set 
$\mathbb{T}^N_{\pi}=\{\vartheta\in\mathbb{T}^N:|\vartheta_k-\vartheta_i|\neq \pi, \; \forall (k,i)\in\mathcal{E}\}$. Thus, by Theorem \ref{general thm}, all trajectories from any connected component $S$ of $\mathbb{T}^N_{\pi}$ will converge to a unique integral curve of $\boldsymbol{1}_{\vartheta}$ in $S$. Such an attractor corresponds to a phase-locking behavior in which the frequencies $\dot{\vartheta_i}$ synchronize to a particular frequency which is unique for any connected component of $\mathbb{T}^N_{\pi}$.

Typically, the set $\mathbb{T}^N_{\pi}$ consists of a finite number of distinct connected components which depends on the communication graph $\mathcal{G}$. Given any initial configuration in $\mathbb{T}^N_{\pi}$, the trajectory is attracted to a limit cycle corresponding to a phase-locking behavior that is unique to the particular connected component of $\mathbb{T}^N_{\pi}$ in which the configuration is found. In the particular case of a tree graph, the set $\mathbb{T}^N_{\pi}$ is itself connected since given any point there exists a continuous path in $\mathbb{T}^N_{\pi}$ connecting each vertex to its parent and thus any point to any element of the synchronization manifold $\mathcal{M}_{\operatorname{sync}}$. Therefore, in the case of a tree graph, for instance, we see that there exists a unique integral curve of the Perron-Frobenius vector field $\boldsymbol{1}_{\vartheta}$ that is an attractor for every point in $\mathbb{T}^N_{\pi}$. That is, we achieve almost global frequency synchronization. The same result holds for any communication graph which results in a connected $\mathbb{T}^N_{\pi}$. If the frequencies $w_k=0$ for all $k$, then this would correspond to synchronization of the agents to a point on $\mathcal{M}_{\operatorname{sync}}$.

\subsection*{Example: formation control} Consider $N$ agents evolving on $\mathbb{S}^1$ and interacting via a connected bidirectional graph according to $\dot{\vartheta}_k=\omega_k + \sum_{i:(k,i)\in\mathcal{E}}f_{ki}(\vartheta_i-\vartheta_k)$,
where each coupling function $f_{ki}$ is odd, $2\pi$-periodic and differentiable on $(0,2\pi)$, and satisfies $f_{ki}(\pi)=0$, $f'_{ki}(\alpha)>0$ for $\alpha \in (0,2\pi)$ and $f_{ki}(\alpha)\rightarrow -\infty$ as $\alpha \rightarrow 0^+$ as in the example depicted in Figure~\ref{fig:couplings} $(b)$. 
In this model, all agents $\vartheta_k$ repel each other with a strength that monotonically decreases on $(0,\pi)$ and grows infinitely strong as the separation between any pair of connected agents approaches $0$. 
Thus, it is clear that the set 
$\mathcal{U}=\{\vartheta\in\mathbb{T}^N:|\vartheta_k-\vartheta_i|\neq0,\; i,k=1,\cdot\cdot\cdot, N\}$,
is forward invariant. Strict differential positivity on $\mathcal{U}$ ensures that all trajectories from any connected component $S$ of $\mathcal{U}$ converge to a unique integral curve of $\boldsymbol{1}\vert_{\vartheta}$ in $S$. In this model, all attractors of the system correspond to balanced formations. 

\begin{figure}
\centering
\includegraphics[width=0.8\linewidth]{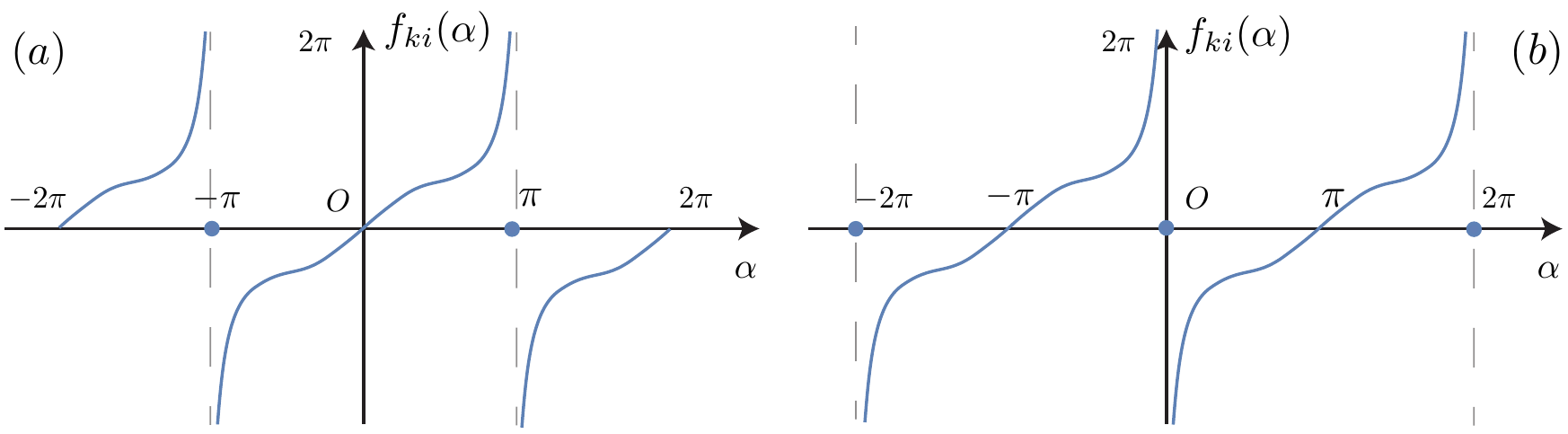}
  \caption{$(a)$ An example of a type of coupling function on $\mathbb{S}^1$ which ensures convergence to a unique limit cycle arising as an integral curve of the Perron-Frobenius vector field $\boldsymbol{1}\vert_{\vartheta}$ in any connected component of $\mathbb{T}^{N}_{\pi}$.
   $(b)$ An example of a type of coupling function that can be used to achieve convergence to balanced formations on $\mathbb{S}^1$.
  }
  \label{fig:couplings}
\end{figure}

\begin{remark}
It should be noted that it may generally be difficult to determine the forward invariant regions $S$ of Theorem \ref{general thm} on which strict differential positivity holds. In the preceding frequency synchronization and formation control examples, we have characterized the forward invariant sets by the use of barrier functions, which effectively split the $N$-torus into a collection of such  sets that is determined by the topology of the communication graph.
\end{remark}

\subsection{The symmetric setting}

In this section we consider consensus protocols on $\mathbb{S}^1$ for undirected graphs and a coupling function $f$ that is independent of the communication edge $(k,i)\in\mathcal{E}$. In this case, the resulting dynamics can be formulated as a gradient system on $G$ and studied using standard tools such as quadratic Lyapunov theory. Here we show that a very similar analysis can be performed via differential positivity through invariant quadratic cone fields, and while this analysis extends to the asymmetric setting of Section \ref{diff1} by simply replacing quadratic cones with polyhedral cones, the quadratic Lyapunov approach does not offer a natural extension to the asymmetric setting.

Consider the model
$
\dot{\vartheta}_k=\sum_{i:(k,i)\in\mathcal{E}} f(\vartheta_i-\vartheta_k)$,
where the coupling function $f$ is an odd, $2\pi$-periodic, and twice differentiable function on $(-\pi,\pi)$, which satisfies $f(0)=0$, and $f'(\alpha)>0$ for all $\alpha\in(-\pi,\pi)$.
We define the cone field $\mathcal{K}_{\mathbb{T}^N}$ on the $N$-torus $\mathbb{T}^N$ by
\begin{equation} \label{consensus4}
\mathcal{K}_{\mathbb{T}^N}(\vartheta,\delta\vartheta):=\{\delta\vartheta\in T_{\vartheta}\mathbb{T}^N:\boldsymbol{1}_{\vartheta}^T\delta\vartheta\geq 0,\; Q(\vartheta,\delta\vartheta)\geq 0\},
\end{equation}
where $Q$ is the quadratic form 
\begin{align}
Q(\vartheta,\delta\vartheta)=\delta\vartheta^T\boldsymbol{1}_{\vartheta}\boldsymbol{1}_{\vartheta}^T\delta\vartheta-\mu\;\delta\vartheta^T\delta\vartheta  
\end{align}
where the constant parameter $\mu\in(0,N)$ controls the opening angle of the cone. Note that (\ref{consensus4}) is clearly a bi-invariant cone field since the defining inequalities are based on a bi-invariant frame and a bi-invariant vector field $\boldsymbol{1}_{\vartheta}$ on $\mathbb{T}^N$.
In the limiting cases, (\ref{consensus4}) defines an invariant half-space field for $\mu=0$, and an invariant ray field for $\mu=N$.

Fix $\mu\in(0,N)$. On the boundary $\partial\mathcal{K}$ of the cone $\mathcal{K}(\vartheta,\delta\vartheta)$, we have 
$
Q(\delta\vartheta):=\delta\vartheta^T\boldsymbol{1}_{\vartheta}\boldsymbol{1}_{\vartheta}^T\delta\vartheta-\mu\;\delta\vartheta^T\delta\vartheta=0
$
and so any $\delta\vartheta\in\partial\mathcal{K}\setminus\{0\}$ satisfies the relation
\begin{equation}
\mu=\frac{\delta\vartheta^T\boldsymbol{1}_{\vartheta}\boldsymbol{1}_{\vartheta}^T\delta\vartheta}{\delta\vartheta^T\delta\vartheta}=\frac{\left(\delta\vartheta_1+ \cdot\cdot\cdot + \delta\vartheta_N\right)^2}{\delta\vartheta_1^2 + \cdot\cdot\cdot + \delta\vartheta_N^2}\geq 0.
\end{equation}
Using $A(\vartheta)\boldsymbol{1}_{\vartheta}=0$, we see that the time derivative of $Q(\delta\vartheta)$ along the trajectories of the variational dynamics satisfies
\begin{align}
{d\over dt} Q(\delta\vartheta)=-\mu\;\delta\vartheta^T(A(\vartheta)^T+A(\vartheta))\delta\vartheta &= -2\mu\;\delta\vartheta^T A(\vartheta)\delta\vartheta \nonumber \\
&=\mu\sum_{(i,j)\in\mathcal{E}} f'(\vartheta_i-\vartheta_j)\;(\delta\vartheta_i-\delta\vartheta_j)^2  \nonumber \\
&=\frac{\left(\delta\vartheta^T\boldsymbol{1}_{\vartheta}\right)^2}{\delta\vartheta^T\delta\vartheta}\sum_{(i,j)\in\mathcal{E}} f'(\vartheta_i-\vartheta_j)\;(\delta\vartheta_i-\delta\vartheta_j)^2
\end{align}
for $\delta\vartheta\in\partial\mathcal{K}\setminus\{0\}$. Thus, we see that on the boundary $Q(\delta\vartheta)=0$, we have ${d\over dt} Q(\delta\vartheta)>0$ for all
$
\vartheta\in\mathbb{T}^N_{\pi}=\{\vartheta\in\mathbb{T}^N:|\vartheta_k-\vartheta_i|<{\pi},\, i,k=1,\ldots,N\}$.
That is, for $\vartheta\in\mathbb{T}^N_{\pi}$, the system is strictly differentially positive with respect to the invariant cone field $\mathcal{K}_{\mathbb{T}^N}$. Note that the invariant vector field 
$
w(\vartheta)={1\over\sqrt{N}}\boldsymbol{1}_{\vartheta}\in\operatorname{int}\mathcal{K}_{\mathbb{T}^N}(\vartheta,\delta\vartheta)
$
 clearly satisfies the conditions in Theorem \ref{thm compact}.

\begin{figure}
\centering
\includegraphics[width=0.7\linewidth]{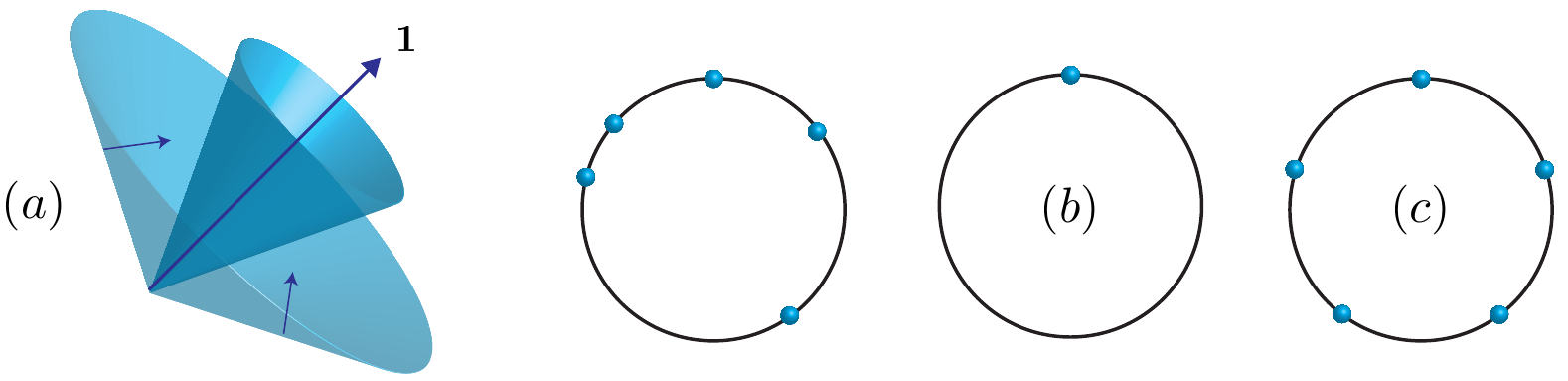}
  \caption{$(a)$ Strict differential positivity of $\dot{\vartheta}_k=\sum_{i:(k,i)\in\mathcal{E}} f(\vartheta_i-\vartheta_k)$ on $\mathbb{T}^N_{\pi}$ with respect to the invariant cone field $\mathcal{K}_{\mathbb{T}^N}$ for any choice of the parameter $\mu$. $(b)$ A synchronized state with $\vartheta\in\mathcal{M}_{\operatorname{sync}}$, and $(c)$ a balanced splay configuration for a network of $N=5$ agents evolving on $\mathbb{S}^1$.
  }
  \label{fig:balanced}
\end{figure}

It follows that the modified consensus dynamics, where the $2\pi$-periodic coupling function $f$ is strictly increasing on $(-\pi,\pi)$ and discontinuous at $\pi$,
is almost everywhere strictly differentially positive on $\mathbb{T}^N$, and thus almost globally monotonically asymptotically convergent to an attractor that arises as an integral curve of the Perron-Frobenius vector field $w$. However, since the dynamics is almost everywhere strictly differentially positive on $\mathbb{T}^N$,  and not everywhere strictly differentially positive due to the discontinuity of the coupling function at $\pi$, the attractor for the system may not be unique. For any given bounded, connected, and forward invariant set in $\mathbb{T}^N$, Theorem \ref{thm compact} establishes convergence to a unique limit cycle contained in the set. However, due to the discontinuity of the coupling function at $\pi$, both the synchronized and balanced configurations could be stable attractors for different bounded, connected, and forward invariant sets.
In particular, almost every trajectory will converge monotonically to either a synchronized state given by the Perron-Frobenius curve through the identity element $e\in\mathbb{T}^N$, or to a balanced configuration arising as the Perron-Frobenius integral curve through a point corresponding to a balanced state. For example, a balanced configuration for a consensus model on $\mathbb{S}^1$ with a complete communication graph is given by $\vartheta_k=2\pi k/N$. See Figure~\ref{fig:balanced} for an illustration of synchronized and balanced configurations in the $N=5$ case.

\section{Differential positivity and consensus on $SO(3)$}  \label{SO(3)}

Let $G$ be a compact Lie group with a bi-invariant Riemannian metric. Consider a network of $N$ agents $g_k$ represented by an undirected connected graph $\mathcal{G}=(\mathcal{V},\mathcal{E})$ evolving on $G$. For a given element $g_k\in G$, the Riemannian exponential and logarithm maps are denoted by $\exp_{g_k}:T_{g_k}G\rightarrow G$ and $\log_{g_k}:U_{g_k}\rightarrow T_{g_k}G$, respectively, where $U_{g_k}\subset G$ is the maximal set containing $g_k$ for which $\exp_{g_k}$ is a diffeomorphism. For any communication edge $(k,i)\in\mathcal{E}$ define 
\begin{equation}
\theta_{ki}=d(g_k,g_i) \quad \operatorname{and} \quad u_{ki}=\frac{\log_{g_k}g_i}{\|\log_{g_k}g_i\|},
\end{equation}
where $d$ denotes the Riemannian distance on $G$. Let $\operatorname{inj}(G)$ denote the injectivity radius of $G$.
A class of consensus protocols can be defined on $G$ by
\begin{equation}  \label{4protocol}
\dot{g}_k=\sum_{i:(k,i)\in\mathcal{E}}f(\theta_{ki})\;u_{ki},
\end{equation}
where $f:[0,\operatorname{inj}(G)]\rightarrow \mathbb{R}$ is a suitable reshaping function that is differentiable on $(0,\operatorname{inj}(G))$ and satisfies $f(0)=0$.

Here we consider a network of $N$ agents evolving on the space of rotations $SO(3)$.  Associate to each agent a state $g_k\in SO(3)$, where 
\begin{equation}
SO(3)=\{R\in\mathbb{R}^{3\times 3}:R^TR=I, \; \det{R}=1\},
\end{equation}
and $e=I$ denotes the identity element and matrix in $SO(3)$. The Lie algebra of $SO(3)$ is the set of $3\times 3$ skew symmetric matrices, and is denoted by $\mathfrak{so}(3)$. For any tangent vector $\delta g_k\in T_{g_k}G$, there exist $\omega_1,\omega_2,\omega_3\in\mathbb{R}^3$ such that 
\begin{equation}
\delta g_{k}= g_k
\begin{pmatrix}
0 & -\omega_3 & \omega_2 \\
\omega_3 &  0 & -\omega_1 \\
-\omega_2 & \omega_1 & 0
\end{pmatrix}=g_{k}\Omega,
\end{equation}
where $\Omega\in\mathfrak{so}(3)$. We can thus identify any 
$\delta g_k\in T_{g_k}G$ with a vector $\Omega^{\vee}=(\omega_1,\omega_2,\omega_3)^T\in\mathbb{R}^3$  via the $^{\vee}$ map. 

For simplicity, assume that $SO(3)$ is equipped with the standard bi-invariant metric characterized by $\langle \Omega_1,\Omega_2 \rangle_{\mathfrak{so}(3)}=\left(\Omega_1^{\vee}\right)^T\Omega_2^{\vee}$, for $\Omega_1,\Omega_2\in\mathfrak{so}(3)$. Consider a consensus protocol of the form 
\begin{equation}  \label{consensus example}
\dot{g}_k=g_{k}\Omega_k+\sum_{i:(k,i)\in\mathcal{E}}f(\theta_{ki})u_{ki},
\end{equation}
where each $\Omega_k\in \mathfrak{so}(3)$ is a constant left-invariant `intrinsic velocity' associated to agent $g_k$, and
 the reshaping function $f:[0,\pi]\rightarrow\mathbb{R}$ is differentiable on $(0,\pi)$, satisfies $f(0)=0$ and $f'(\theta)>0$ for $\theta\in(0,\pi)$. Thus, in the absence of the flow of communication between the agents, each agent would evolve according to
 \begin{equation}
 g_k(t)=g_k(0)e^{t\Omega_k}.
 \end{equation}
 
Let $\nabla$ denote the left Cartan connection on $SO(3)$. We arrive at the linearization of (\ref{consensus example})  with respect to a left-invariant frame by considering
\begin{equation}
\nabla_{\delta g_k}\left(g_k\Omega_k+\sum_{i:(k,i)\in\mathcal{E}}f(\theta_{ki})u_{ki}\right),
\end{equation}
which is a measure of the change in $g_k\Omega_k+\sum_{i:(k,i)\in\mathcal{E}}f(\theta_{ki})u_{ki}$ when $g_k$ changes infinitesimally in the direction of $\delta g_k\in T_{g_k}G$. As $g_k\Omega_k$ is a left-invariant vector field and $\nabla$ is the left Cartan connection, it immediately vanishes for each $k$.

Let $B(g_k)$ be a geodesic ball centered at $g_k\in G$ and $r:B(g_k)\rightarrow \mathbb{R}$ the function returning the geodesic distance to $g_k$, and $\gamma: [0,\theta_{ki}]\rightarrow B(g_k)$ be the unit speed geodesic connecting $g_k$ to $g_i\in B(g_k)$. 
The term corresponding to the communication edge $(k,i)\in\mathcal{E}$ in the 
linearization can be calculated by considering the expression $\nabla_{\xi}(f(r)\frac{\partial}{\partial r})$, where $\xi \in T_{g_i}G$ and $\frac{\partial}{\partial r}$ denotes the normalized radial vector field in normal coordinates centered at $g_k$. Decomposing $\xi$ into $\xi=\xi^{\parallel}+\xi^{\perp}$, where $\xi^{\parallel}$ is parallel to $\frac{\partial}{\partial r}$ and $\xi^{\perp}$ is orthogonal to it, we obtain by the properties of affine connections:
\begin{align}
\nabla_{\xi}\left(f(r)\frac{\partial}{\partial r}\right) &= \xi(f(r))\frac{\partial}{\partial r}+f(r)\nabla_{\xi^{\parallel}}\frac{\partial}{\partial r}+f(r)\nabla_{\xi^{\perp}}\frac{\partial}{\partial r} \\
& = \xi(f(r))\frac{\partial}{\partial r}+f(r)\nabla_{\xi^{\perp}}\frac{\partial}{\partial r} \\
&= f'(r)\xi^{\parallel}+f(r)\nabla_{\xi^{\perp}}\frac{\partial}{\partial r}, \label{SO(3) lin1}
\end{align}
where in the second line we have used the fact that $\partial/\partial r$ is tangent to unit speed geodesics emanating from $g_k$ to write $\nabla_{\partial/\partial r}\partial/\partial r =0$, and thus $\nabla_{\xi^{\parallel}}\frac{\partial}{\partial r}=0$. 

Now let $\zeta$ be a curve in the geodesic sphere $S_{ki}:=\{g\in G: r(g)=\theta_{ki}\}$ with $\zeta(0)=g_i$ and $\zeta'(0)=\xi^{\perp}$. In normal coordinates centered at $g_k$, consider a smooth geodesic variation $\Gamma:(-\epsilon,\epsilon)\times [0,\theta_{ki}]\rightarrow G$ for some $\epsilon >0$ with $\Gamma(0,t)=\gamma(t)$ and $\Gamma(s,\theta_{ki})=\zeta(s)$. In terms of the vector fields $\partial_t\Gamma(s,t)$ and $\partial_s\Gamma(s,t)$, the corresponding Jacobi field \cite{Grove1974} $J:G\rightarrow TG$ takes the form
$J(t)=\partial_s\Gamma(0,t)$. Using the definition of the torsion tensor in (\ref{torsion def}) and 
$[\partial_t\Gamma,\partial_s\Gamma](s,t)=d\Gamma\vert_{(s,t)}([\partial/\partial t,\partial/\partial s])=0$ at the point $g_i$, we find that
\begin{align}
\nabla_{\xi^{\perp}}\frac{\partial}{\partial r}&=\nabla_{\partial_s\Gamma}\partial_t\Gamma(0,r) = 
\nabla_{\partial_t\Gamma}\partial_s\Gamma(0,r) + T(\partial_s\Gamma(0,r),\partial_t\Gamma(0,r)) \\
&= \nabla_{\xi^{\parallel}}J(r)+T(J,\xi^{\parallel}),  \label{SO(3) lin2}
\end{align}
where $T$ is the torsion tensor of the left Cartan connection on $SO(3)$.

Now the Jacobi field equation of an affine connection $\nabla$ with torsion $T$ is given by
\begin{equation} \label{Jacobi eq}
\nabla_{\gamma'}\nabla_{\gamma'}J+\nabla_{\gamma'}\left(T(J,\gamma')\right)+R(J,\gamma')\gamma'=0,
\end{equation}
where $\gamma'$ denotes the tangential field of the geodesic $\gamma$ \cite{Grove1974}. Since the torsion is covariantly constant for the left Cartan connection, we have
\begin{equation}
\nabla_{\gamma'}\left(T(J,\gamma')\right)= T(\nabla_{\gamma'}J,\gamma')+T(J,\nabla_{\gamma'}\gamma'),
\end{equation}
and the second term $T(J,\nabla_{\gamma'}\gamma')=0$ since $\gamma$ is an affine geodesic. Thus, noting that the curvature of the left Cartan connection is null, we find that the Jacobi field equation (\ref{Jacobi eq}) reduces to 
\begin{equation}
\nabla_{\gamma'}\nabla_{\gamma'}J+T(\nabla_{\gamma'}J,\gamma') = 0.
\end{equation}
Describing the Jacobi field along the geodesic $\gamma$ generated by $\gamma'(0)$ through a map $J:\mathbb{R}\rightarrow \mathfrak{g}$ and left translation along $\gamma$, we obtain:
\begin{equation}
J''(t)+[\gamma'(0),J'(t)]=0.
\end{equation}
We form an orthonormal basis $\{\boldsymbol{u}_1,\boldsymbol{u}_2,\boldsymbol{u}_3\}$ of
 $T_{g_i}G$, where $\boldsymbol{u}_1= \xi^{\parallel}/\|\xi^{\parallel}\|$ and uniquely extend it to a left-invariant frame on $G$.
The Jacobi field equation along the geodesic with unit tangent $\partial/\partial r$ now takes the form
\begin{equation} \label{Jacobi 3}
\begin{cases}
J_1''(t)= 0, \\
J_2''(t)-J_3'(t) = 0, \\
J_3''(t) + J_2'(t) = 0,
\end{cases}
\end{equation}
where $J_i(t)$ denote the components of $J(t)$ with respect to this left-invariant frame. Note that we have used the Lie bracket on $SO(3)$ to derive the form of the equations in (\ref{Jacobi 3}).
It can easily be verified that the solution to this system subject to $J(0)=0$ and $J(r)=\xi^{\perp}=\xi^{\perp}_2\boldsymbol{u}_2+\xi^{\perp}_3\boldsymbol{u}_3$ satisfies $J_1(r)=0$ and 
\begin{align}
J_2'(r)-J_3(r)&=\frac{1}{2}\cot\left(\frac{r}{2}\right) \xi_2^{\perp} + \frac{1}{2}\xi_3^{\perp} \\
J_3'(r)+J_2(r) &= -\frac{1}{2}\xi_2^{\perp} + \frac{1}{2}\cot\left(\frac{r}{2}\right)\xi_3^{\perp}.
\end{align}
Substitution into (\ref{SO(3) lin2}) and (\ref{SO(3) lin1}) finally yields
$\nabla_{\xi}(f(r)\partial/\partial r)=\hat{A}(r)\xi$, where $\hat{A}(r)$ is a linear operator whose spectrum is given by
\begin{equation} \label{spectrum}
\bigg\{f'(r),\frac{f(r)}{2}\left(\cot\left(\frac{r}{2}\right)-i\right),\frac{f(r)}{2}\left(\cot\left(\frac{r}{2}\right)+i\right)\bigg\}.
\end{equation}
Note that the real parts of the eigenvalues of $\hat{A}(r)$ are all positive for $r\in (0,\pi)$.

Now writing $\boldsymbol{g}=(g_1,\cdot\cdot\cdot,g_N)$, the dynamical system (\ref{consensus example}) can be expressed as
$
\frac{d}{dt}\boldsymbol{g}=F(\boldsymbol{g})$,
where $F$ is a smooth vector field on $SO(3)^N$. The linearization of the system can be expressed in the form
\begin{equation}
\frac{d}{dt}\mathbf{v}=\mathcal{A}(\boldsymbol{g})\;\mathbf{v},
\end{equation}
where $\mathbf{v}\in\mathbb{R}^{3N}$ is the vector representation of 
$
(\delta g_1,\cdot\cdot\cdot,\delta g_N)\in T_{\boldsymbol{g}}SO(3)^N
$
with respect to a left-invariant orthonormal frame $\{E^l_{k}\}$ $(l=1,2,3; k=1,\cdot\cdot\cdot,N)$ on $T_{\boldsymbol{g}}SO(3)^N\cong \mathbb{R}^{3N}$ formed as the Cartesian product of $N$ copies of a given left-invariant orthonormal frame of $SO(3)$.
For each $\boldsymbol{g}$, the linear map $\mathcal{A}(\boldsymbol{g})$ has the $3N\times 3N$ matrix representation of the form
\begin{equation}
\begin{cases}
\mathcal{A}_{kk}(\boldsymbol{g})&=-\sum_{i: (k,i)\in\mathcal{E}}A(\theta_{ki}), \\
\mathcal{A}_{ki}(\boldsymbol{g})&=A(\theta_{ki}) \quad \mathrm{if} \quad(k,i)\in\mathcal{E}, \\
\mathcal{A}_{ki}(\boldsymbol{g})&=0 \quad \mathrm{if} \quad (k,i)\notin \mathcal{E}.
\end{cases}
\end{equation}
with respect to the orthonormal basis $\{E^l_k\vert_{\boldsymbol{g}}\}_{l=1,2,3; k=1,\cdot\cdot\cdot,N}$ of $\mathbb{R}^{3N}$, where $A(r)$ is the $3\times 3$ matrix representation of the linear operator $\hat{A}(r)$
with spectrum (\ref{spectrum}).

Let $\boldsymbol{1}_{1}$ denote the vector consisting of $N$ copies of $\boldsymbol{e}_1 = (1,0,0)^T$ with respect to the left-invariant frame $\{E^l_k\vert_{\boldsymbol{g}}\}$ of $T_{\boldsymbol{g}}SO(3)^N\cong \mathbb{R}^{3N}$. Similarly, define $\boldsymbol{1}_{2}$ and $\boldsymbol{1}_{3}$ using  $N$ copies of $\boldsymbol{e}_2 = (0,1,0)^T$
 and $\boldsymbol{e}_3 = (0,0,1)^T$, respectively.
 We define the left-invariant cone field $\mathcal{C}_{SO(3)^N}(\boldsymbol{g},\delta\boldsymbol{g})$ of rank 3 by 
\begin{equation}
Q(\mathbf{v}):= \mathbf{v}^T\boldsymbol{1}_{1}{\boldsymbol{1}_1}^T\mathbf{v} + \mathbf{v}^T\boldsymbol{1}_{2}{\boldsymbol{1}_2}^T\mathbf{v} + \mathbf{v}^T\boldsymbol{1}_{3}{\boldsymbol{1}_3}^T\mathbf{v} 
-\mu\;\mathbf{v}^T\mathbf{v} \geq 0,
\end{equation}
where $\mu\in (0,3N)$ is a parameter and $\mathbf{v}$ is the vector representation of $\delta\boldsymbol{g}\in T_{\boldsymbol{g}}SO(3)^N$ with respect to the left-invariant frame $\{E^l_k\vert_{\boldsymbol{g}}\}$ of $T_{\boldsymbol{g}}SO(3)^N\cong \mathbb{R}^{3N}$. Observe that 
\begin{equation}  \label{eigenspace}
\mathcal{D}_{\boldsymbol{g}}=\operatorname{span}\{\boldsymbol{1}_{1},\boldsymbol{1}_{2},\boldsymbol{1}_{3}\}\subset\operatorname{int} \mathcal{C}(\boldsymbol{g}),
\end{equation}
at every point $\boldsymbol{g}\in SO(3)^N$. Furthermore, since $\mathcal{D}_{\boldsymbol{g}}$ is a left-invariant distribution of rank $3$ that defines a Lie subalgebra of $\mathfrak{so}(3)^N$ isomorphic to $\mathfrak{so}(3)$, its integral manifolds correspond to left translations of the $3$-dimensional subgroup of $SO(3)^N$ obtained as $N$ identical copies of $SO(3)$ diagonally embedded in $SO(3)^N$. In particular, the synchronization manifold
\begin{equation}
\mathcal{M}_{\operatorname{sync}}=\{\boldsymbol{g}\in SO(3)^N: g_1= \cdot\cdot\cdot = g_N\}
\end{equation}
is the integral manifold through the identity element $\boldsymbol{e}\in\ SO(3)^N$.

Noting that $\mathcal{A}(\boldsymbol{g})\boldsymbol{1}_{j}=0$, for $j=1,2,3$, we find that the time derivative of $Q$ along trajectories of the variational dynamics takes the form
\begin{equation}
\frac{d}{dt}Q(\mathbf{v})=-2\mu\;\mathbf{v}^T\mathcal{A}(\boldsymbol{g})\mathbf{v} 
=  \mu \sum_{(k,i)\in\mathcal{E}}(\mathbf{v}_k-\mathbf{v}_i)^T A(\theta_{ki})(\mathbf{v}_k-\mathbf{v}_i) \geq 0,
\end{equation}
where $\mathbf{v}_k\in\mathbb{R}^3$ consists of the elements of $\mathbf{v}\in\mathbb{R}^{3N}$ at the entries $3k-2$, $3k-1$, $3k$. It is clear that for a connected graph $dQ/dt>0$, unless $\mathbf{v}_i=\mathbf{v}_k$ for all $i,k$. This demonstrates strict differential positivity of the consensus dynamics with respect to the cone field $\mathcal{C}$ for the monotone coupling function $f$, whenever $\theta_{ki}<\pi$.
Thus, by Theorem \ref{thm compact2}, for any bounded, connected, and forward invariant region $S\subseteq SO(3)^N$, there exists a unique integral manifold of $\mathcal{D}_{\boldsymbol{g}}$ (\ref{eigenspace}) that is an attractor for all of the trajectories from $S$. 
In particular, if $\Omega_k=0$ for all $k=1,...,N$ then the set $S=\{\boldsymbol{g}\in SO(3)^N: d(g_i,g_k) < \pi/2, \forall(i,k)\in\mathcal{E}\}$ is forward invariant and thus contains a unique attractor that is an integral of the distribution $\mathcal{D}_{\boldsymbol{g}}$. On the other hand, since $e\in S$ is an equilibrium point the attractor must be the integral manifold through $e$, which coincides with the three-dimensional synchronization manifold $\mathcal{M}_{\operatorname{sync}}\cong SO(3)$.

Imposing the additional requirement on $f$ that $f(\theta)\rightarrow \infty$ as $\theta \rightarrow \pi$ effectively splits $SO(3)^N$ into a collection of connected and forward invariant regions on which the dynamics is strictly differentially positive, regardless of the distribution of `intrinsic left-invariant velocities' $\Omega_k\in\mathfrak{so}(3)$ of the agents $g_k$. Thus, for any initial condition, the agents will asymptotically converge to a synchronized left-invariant motion in which the relative position of the agents $g_k$ are fixed. The particular asymptotic configuration of the agents and their collective left-invariant motion will be unique to the forward invariant set to which the initial condition belongs.

\section{Conclusion}

We have presented a detailed framework for studying differential positivity of discrete and continuous-time dynamical systems on Lie groups using invariant cone fields. Throughout the paper, we have illustrated the relevant concepts by applying the theory to  consensus theory, including examples involving asymmetric couplings and inhomogeneous dynamics. We have also reviewed a generalization of linear positivity theory that is obtained when one replaces the notion of a dominant eigenvector with that of a dominant eigenspace of dimension $k$. For such systems, it is natural to characterize positivity by the contraction of a cone of rank $k$ in place of a convex cone as in classical positivity theory. The corresponding generalization to nonlinear systems leads to differential positivity with respect to cone fields of rank $k$. The resulting differential Perron-Frobenius theory shows that a distribution of rank $k$ corresponding to dominant modes shapes the attractors of the system. As illustrated with an example concerning consensus on $SO(3)$, this framework can be used to study systems whose attractors arise as integral submanifolds of the distribution.

\bibliographystyle{siamplain}
\bibliography{references}

\end{document}